\documentclass{article}

 \usepackage[preprint]{neurips_2024}

\usepackage[utf8]{inputenc} % allow utf-8 input
\usepackage[T1]{fontenc}    % use 8-bit T1 fonts
\usepackage{hyperref}       % hyperlinks
\usepackage{url}            % simple URL typesetting
\usepackage{booktabs}       % professional-quality tables
\usepackage{amsfonts}       % blackboard math symbols
\usepackage{nicefrac}       % compact symbols for 1/2, etc.
\usepackage{microtype}      % microtypography
\usepackage{xcolor}         % colors
\usepackage[lofdepth,lotdepth]{subfig}
\usepackage{graphicx}
\usepackage{subcaption}
\usepackage{amsmath,amsfonts, amsthm, amssymb, here, dsfont}
\usepackage{enumitem}

\def\argmin{\mathop{\rm argmin}}
\newcommand{\err}{{\rm Err}}

\DeclareMathOperator{\ent}{Ent}
\DeclareMathOperator{\Div}{Div}
\DeclareMathOperator{\re}{\textit{re}}
\DeclareMathOperator{\crps}{CRPS}
\DeclareMathOperator{\sgn}{sgn}
\newcommand{\one}{\mathds{1}}
\newtheorem{theorem}{Theorem}
\newtheorem{proposition}{Proposition}
\newtheorem{cor}{Corollary}
\newtheorem{lemma}{Lemma}

\newtheorem{assumption}{Assumption}
\newtheorem{remark}{Remark}

\title{Distributional regression with reject option}

\author{%
  Cl\'ement Dombry \thanks{\texttt{clement.dombry@univ-fcomte.fr}}\\ %\thanks{\texttt{clement.dombry@univ-fcomte.fr}}\\
  Université Marie et Louis Pasteur, \\
  CNRS, LmB (UMR 6623),\\ 
  F-25000 Besan\c{c}on, France.\\
  \texttt{clement.dombry@univ-fcomte.fr} \\
  \And
  Ahmed Zaoui \thanks{\texttt{ahmed.zaoui@univ-fcomte.fr}}\\
  Université Marie et Louis Pasteur, \\
  CNRS, LmB (UMR 6623),\\ 
  F-25000 Besan\c{c}on, France.\\
  \texttt{ahmed.zaoui@univ-fcomte.fr} \\ 
}

\begin{document}

\maketitle

\begin{abstract}
Selective prediction, where a model has the option to abstain from making a decision, is crucial for machine learning applications in which mistakes are costly. In this work, we focus on distributional regression and introduce a framework that enables the model to abstain from estimation in situations of high uncertainty. We refer to this approach as \emph{distributional regression with reject option}, inspired by similar concepts in classification and regression with reject option. We study the scenario where the rejection rate is fixed. We derive a closed-form expression for the optimal rule, which relies on thresholding the entropy function of the Continuous Ranked Probability Score (CRPS). We propose a semi-supervised estimation procedure for the optimal rule, using two datasets: the first, labeled, is used to estimate both the conditional distribution function and the entropy function of the CRPS, while the second, unlabeled, is employed to calibrate the desired rejection rate. Notably, the control of the rejection rate is distribution-free. Under mild conditions, we show that our procedure is asymptotically as effective as the optimal rule, both in terms of error rate and rejection rate. Additionally, we establish rates of convergence for our approach based on distributional k-nearest neighbor. A numerical analysis on real-world datasets demonstrates the strong performance of our procedure.
\end{abstract}

%%%%%%%%%%%%%%%%%%%%%%%%%%%%%%%%%%%%%%%%%%%%%%%%%%%%%%%%%%%%%%%%%%%%%%
%%%%%%%%%%%%%%%%%%%%%%%%%%%%%%%%%%%%%%%%%%%%%%%%%%%%%%%%%%%%%%%%%%%%%%
\section{Introduction} \label{sec:intro}
%%%%%%%%%%%%%%%%%%%%%%%%%%%%%%%%%%%%%%%%%%%%%%%%%%%%%%%%%%%%%%%%%%%%%
Prediction is a central task in statistics and machine learning. The most common approach in statistics is regression, which aims to produce a point prediction of the target variable $Y\in \mathbb{R}$ based on covariates $X\in \mathbb{R}^d$. Numerous algorithms exist for point regression~\citep{Gyofri_Kohler_Krzyzak_Walk02, Tsybakov08}, and their theoretical properties are well studied. However, a fundamental limitation of point regression is that it does not provide a measure of prediction uncertainty, which is also crucial in the decision-making process.\\
\\Distributional regression has emerged to address this limitation by modeling the variability of the predicted variable $Y$. The conditional distribution of $Y$ given $X$ is estimated, providing a complete representation of the range of possible outcomes given the information in the covariates. Unlike point regression, which focuses on a single predicted value, distributional regression produces predictions that capture both the general trend and the associated uncertainty. This approach is particularly relevant in contexts where data exhibit high variability or complex behaviors. Distributional regression plays a crucial role in a variety of applied fields, offering valuable insights and predictive power. Notable applications include statistical post-processing of weather forecasts~\citep{Matheson_Winkler76, Gneiting05}, forecasting wind gusts~\citep{Baran2015} and solar irradiance~\citep{Schulz2021}, as well as predicting ICU length of stays during the COVID-19 pandemic (Henzi et al., 2021a). Another recent application of distributional regression is in the fields of medicine and oncology, for the estimation of breast cancer ODX scores with uncertainty quantification~\citep{AlMasry2024}.\\
\\ Distributional regression models, in particular, generally rely on the minimization of a proper scoring rule to align the predictive distribution with the actual observations. Among these scoring rules, the Continuous Ranked Probability Score (CRPS, ~\citealp{Matheson_Winkler76, Gneiting_Raftery07}) is the most widely used. The CRPS evaluates the difference between actual observations and predictive distribution, providing a robust measure to evaluate how well a model captures the true variability and uncertainty inherent in the data. A variety of advanced post-processing techniques have been introduced to improve predictive modeling and uncertainty quantification. Among these, methods such as distributional k-nearest neighbors (KNN, \citealp{ZT89}), Ensemble Model Output Statistics (EMOS, \citealp{Gneiting_Raftery07}), Isotonic Distributional Regression \citep{Henzi_et_al_2021}, Distributional Regression Network (DRN, \citealp{Rasp_Lerch18}) or Distributional Random Forest (DRF, \citealp{Cevid_MichelNBM22}). The theory of distributional regression via scoring rule minimization has recently gained attention. ~\citet{Pic_Dombry_Naveau_Taillardat23} studied the minimax convergence rates with the CRPS error for distribution regression. \citet{Dombry_Zaoui24} provided statistical learning guarantees for model fitting, model selection, and convex aggregation based on CRPS minimization.\\
 \\Despite the development of new algorithms, the richness of information in distributional regression can also present challenges, particularly the presence of outlier predictions or extreme values that may distort the conclusions drawn from the model. To address these issues, integrating  reject option is crucial. This method aims to filter out unreliable predictions, allowing us to focus on the most relevant and robust result. Inspired by the concept of regression with reject option as discussed by~\citet{Zaoui2020}, we extend this framework to distributional regression.\\
\\The study of learning with reject option is a pertinent area of research where the main focus is to abstain from making predictions when there is significant uncertainty about the predicted value; see, for instance,~\citealp{Chow70,Chow57,Herbei_Wegkamp06, denis_hebiri_2020, Zaoui2020, chzhen_al2021} and references therein. This approach helps to reduce prediction errors by offering an alternative: Instead of making a potentially risky prediction, the model can choose to "reject" the decision and not provide a result. The reject option was first explored in~\citet{Chow57} and gained traction within the statistical learning community in the early 2000s, particularly with the development of conformal prediction in~\cite{Vovk02IndepError,Vovk99IntroCP,Vovk_Gammerman_Shafer05}. Three main approaches can be considered, with none having a clear advantage over the others: the first aims to ensure a predefined level of coverage, the second to achieve a predefined rejection rate, and the third to balance the two through a trade-off. These approaches have been rigorously analyzed in the literature, which highlights both the theoretical guarantees and practical effectiveness of each method under various conditions.\\
\\\textbf{Contributions.} In this work, we focus on the reject option in the distributional regression setting, where the rejection rate is predefined (controlled). Our main contributions are the following:
\begin{enumerate}
\item We first introduce the framework of distributional regression with reject option, inspired by regression with reject option~\citep{Zaoui2020}. The optimal predictor subject to a bounded rejection rate is derived and we show that the reject option must be used if the entropy function associated with the CRPS exceeds some explicit threshold. 
\item We propose a plug-in approach to estimate the optimal predictor based on a two-step procedure. The first step estimates the conditional distribution function using a labeled dataset. The second step determines the entropy associated with the CRPS using a second unlabeled dataset and estimates the threshold above which the reject option must be used. The main advantage of our procedure is that it can be applied to any estimator of the conditional distribution.
\item We study the asymptotic properties of this approach under mild assumptions and show in particular consistency,  showing that for large samples our method achieves  the same performance as the optimal predictor, both in terms of error rate and rejection rate.
\item In the particular case of the $k$-nearest neighbour estimator of the conditional distribution, we provide non asymptotic concentration bounds for the excess of risk of the proposed method.
\end{enumerate}
\textbf{Outline of the paper.} In Section~\ref{sec:generalframework}, we formally state the distributional regression with reject option. We provide an optimal rule with reject option and propose a strategy for estimating it. In Section~\ref{sec:statisticalguarantees}, we investigate its statistical properties: its consistency in a general setting on the one hand, and on the other hand, we determine its rate of convergence by applying our approach to distributional $k$-nearest neighbors. Finally, in Section~\ref{sec:numerical}, we present its numerical performance on real datasets.\\
\\\textbf{Notations.} Let $\mathcal{P}(\mathbb{R})$ denote the set of all probability measures on $\mathbb{R}$ and $\mathcal{P}_1(\mathbb{R})$ the subset of probability measures with finite absolute moment, that is the Wasserstein space 
\begin{equation*}\label{eq:def-P1}
\mathcal{P}_1(\mathbb{R})=\Big\{F\in\mathcal{P}(\mathbb{R})\colon \int_{\mathbb{R}}|y|F(\mathrm{d}y)<\infty\Big\}.
\end{equation*}
We denote by $W_1$ the Wasserstein distance of order $1$ on the space $\mathcal{P}_1(\mathbb{R})$ 
given by
\[
W_{1}(F,G)=\int_{\mathbb{R}}\left|F(y)-G(y)\right|dy.
\]
For a monotone increasing function $T$ we denote by $T^{-1}$ it generalized inverse defined for all $p\in (0,1)$ as
\[
T^{-1}(p) =\inf\left\{t\in \mathbb{R} : T(t)\geq p\right\}.
\]
\section{General framework}
\label{sec:generalframework}
This section is devoted to presenting the general framework of our study. After introducing the distributional regression model, we briefly describe the Continuous Ranked Probability Score (CRPS) in Section~\ref{sec:crps}. We then introduce the reject option in distributional regression in a general manner in Section~\ref{sec:rejectoption}, and derive the optimal rule with fixed rejection rate in Section~\ref{sec:fixedrejectionrate}. Finally we propose a general data-driven procedure to the optimal rule with reject option in Section~\ref{sec:data-driven procedure}.\\
\\ Let $(X,Y) \in \mathbb{R}^d\times \mathbb{R}$, where $X$ represents the covariates and $Y$ is the corresponding output. We assume that $(X,Y)\sim \mathrm{P}$, for some unknown distribution $\mathrm{P}$. Let $F^{*}_x(y)=\mathbb{P}(Y\leq y|X=x)$ be the conditional cumulative distribution of $Y$ given $X=x$. The objective of the  distributional regression is to estimate  $F^{*}_x$ as a function of $x$.
% We observe a sample $ \mathcal{D}_n=\{(X_i,Y_i)\}_{i=1}^{n}$ of $n$ independent copies of $(X,Y)\in\mathbb{R}^d\times\mathbb{R}$. The predictor uses the training sample \( D_n \) and an algorithm to derive a functional estimator \( \hat{F}_n : x \mapsto \hat{F}_{n,x} \) of the mapping \( F^* : x \mapsto F_{x}^* \). The performance of this estimator is evaluated through its theoretical risk.
To achieve this, we observe a sample $\mathcal{D}_n = \{(X_i, Y_i)\}_{i=1}^n $ consisting of $n$ independent copies of $(X, Y)$. Using this training sample $\mathcal{D}_n$ and a suitable algorithm, we derive a functional estimator $\hat{F}_n : x \mapsto \hat{F}_{n,x}$, which approximates the true mapping $F^* : x \mapsto F^*_x$. The accuracy of this estimator is measured here by its theoretical risk based on the Continuous Ranked Probability Score.
%\subsection{Popular models for distributional regression}
\subsection{The Continuous Ranked Probability Score}
\label{sec:crps}
The Continuous Ranked Probability Score (CRPS, \citealt{Matheson_Winkler76}) is a widely used scoring rule for assessing the quality of probabilistic forecasts. It quantifies the closeness of a forecast distribution $H$ to the observed outcome $y$. Mathematically, it is expressed as:
\[
\crps(H,y) = \int_{\mathbb{R}} \big(H(u) - \mathds{1}_{\{y \leq u\}}\big)^2 \, du,
\]
where $H(u)$ represents the cumulative distribution function of the forecast, and \( \mathds{1}_{\{y \leq u\}} \) is an indicator function that equals $1$ if \( y \leq u \) and 0 otherwise.

The Continuous Ranked Probability Score (CRPS) can be applied to any predictive distribution $H$ with finite absolute moments. A lower CRPS value indicates superior forecast performance, signifying that the predicted distribution more closely aligns with the observed outcomes. Consequently, the CRPS is a negatively oriented score.

The average CRPS can be decomposed into two distinct components: an entropy term and a divergence term (see Lemma~\ref{lem:expectedCrps}). This decomposition is given by:  
\[
\overline{\crps}(H,K) = \mathbb{E}_{Y \sim K} \left[\crps(H, Y)\right] = \ent(K) +  \Div(H,K),
\]
where 
\begin{align*}
\ent(K) = \int_{\mathbb{R}} K(u) \big(1 - K(u)\big) du,\enspace \text{and}\enspace
\Div(H,K)=  \int_{\mathbb{R}} \big(H(u) - K(u)\big)^2 du,
\end{align*}
denote the entropy and divergence functions respectively. The entropy term captures the inherent uncertainty in the distribution $K$, while the divergence term measures the discrepancy between the predictive distribution $H $ and the reference distribution $K$. The CRPS achieves its minimum when $H=K$, confirming that it is a strictly proper scoring rule~\citep{Gneiting05}.

The following formula, widely used in practice, provides a closed form expression for the CRPS when the predictive distribution is the Gaussian distribution $H=\mathcal{N}(m,\sigma^2)$: then
\[
\crps(H,y)=\sigma\left(z(2\Phi(z)-1)+2\phi(z)-1/\sqrt{\pi}\right),
\]
with $ z=( y-\mu)/\sigma$ and $\Phi$ (resp. $\phi$) denoting the cdf (resp. df) of the standard normal distribution. The case $F$ corresponds to discrete predictive distributions, which are represented by a weighted sample $\left(y_i\right)_{1\leq i\leq n}$ with weights $\left(w_i\right)_{1\leq i\leq n}$, meaning the form $H=\sum_{i=1}^n w_i\delta_{y_i}$ (with $\delta_y$ the Dirac mass at $y$), the $\mathrm{CRPS}$ can be computed~\citep{Gneiting_Raftery07} by
\[
\crps(H,y)= \sum_{i=1}^n w_i|y_i-y| -\frac{1}{2}\sum_{1\leq i< j\leq n}w_iw_j|y_i-y_j|.
\]
\subsection{Distributional regression with reject option}
\label{sec:rejectoption}
Let us start by formalizing the problem. Let $F:\mathbb{R}^d \to \mathcal{P}(\mathbb{R})$ be some measurable distribution function. A predictor with reject option $\Gamma_F$ associated to $F$ is  a measurable function from $ \mathbb{R}^d$ to $\mathcal{P}(\mathbb{R})\cup \{\re\}$ such for $x \in \mathbb{R}^d$, $\Gamma_F(x)\in\{F_x,\re\}$. The use of a predictor with reject option offers two possibilities: either $\Gamma_F(x)=F_x$ and we estimate $F_x$; or $\Gamma_F(x)=\re$ and the reject option has been used. This predictor has two important characteristics which are defined as follows: the rejection rate $r(\Gamma_F)=\mathbb{P}(\Gamma_F(X)=\re)$ and the error rate 
\[
\err(\Gamma_F)=\mathbb{E}\left[\crps(F_,Y)\big| \Gamma_F(X)\neq \re\right].
\]
To measure the performance of $\Gamma_F$, we consider a trade-off between these two quantities and we define the risk of $\Gamma_F$ by 
\[
 \mathcal{R}_{\lambda}(\Gamma_F) = \mathbb{E}\left[\crps(F_X,Y)\one_{\{\Gamma_F(X)\neq \re}\}\right]+\lambda r(\Gamma_F),
\]
where $\lambda >0$ is a parameter controlling the trade-off. We are interested in minimizing the risk $ \mathcal{R}_{\lambda}$ over the measurable functions $F$ and all predictors with reject option $\Gamma_F$ that relies on $F$.  The following proposition provides a closed form of the optimal predictor with reject option.
\begin{proposition} For any $\lambda > 0$, the optimal predictor 
    \[
    \Gamma_{\lambda}^{*}\in \argmin_{\Gamma_F,F}\mathcal{R}_{\lambda}(\Gamma).
    \]
is unique (modulo null probability sets) and is defined by
 \[
    \Gamma^{*}_{\lambda}(X)=\begin{cases}
  F^{*}_X  & \text{if} \;\; \ent(F^*_X) \leq\lambda,\\
  \re     & \text{otherwise}.
  \end{cases}
 % \label{eq:eqOracle}
\]
\label{prop:optimalpredictor}
%\textcolor{red}{I think Assumption~1 is required here for uniqueness.}
\end{proposition}
%\section{Proof}
We can remark that the entropy $\ent(F^*_X)$ plays a central role when deciding whether to use the reject option or not in distributional regression. Whether to make an estimation or not depends on the thresholding of $\ent(F^*_X)$. The following proposition gives the properties of $\Gamma^{*}_{\lambda}$ in terms of error and rejection rate.
\begin{proposition}
For any $0<\lambda \leq \lambda'$, the following holds
%\begin{enumerate}
%\item 
\[
r(\Gamma_{\lambda'}^{*})\leq r(\Gamma_{\lambda}^{*}), \enspace \text{and}\enspace
\err(\Gamma_{\lambda}^{*})\leq \err(\Gamma_{\lambda'}^{*}).
\]
%\end{enumerate}
\label{prop:propertiesoptimalpredictor}
\end{proposition}
%Proposition~\ref{prop:propertiesoptimalpredictor} prove The disadvantage of the formulation of $\Gamma^{*}_{\lambda}$ is the choice of $\lambda$. It is difficult to interpret the value of $\lambda$ in practice since the error and the rejection rate of optimal predictor move in opposite directions with respect to $\lambda$. In this case, we have two possibilities: either we constrain the error or we constrain the rejection rate. In this work, we focus on fixing the rejection rate.
The parameter $\lambda$ can be interpreted as a cost for applying the reject option and Proposition~\ref{prop:propertiesoptimalpredictor} shows that  a larger value of $\lambda$ leads to a lower rejection rate and a higher error. We illustrate this in Figure~\ref{fig:ErrorRejet}, see Section~\ref{sec:drawbackogchoosinglambda} in the supplementary material. The choice of $\lambda$ in practice is not easy and we suggest instead to constrain the rejection rate and to look for the optimal predictor with fixed rejection rate. 

\subsection{Fixed rejection rate}
\label{sec:fixedrejectionrate}
Following the established strategy of regression and classification with reject option ~\citep{Zaoui2020,chzhen_al2021, denis_hebiri_2020}), we predefine the rejection rate. Specifically, we restrict the focus on predictors with  rejection rate bounded by $\varepsilon \in (0,1)$ and aim at minimizing the error. We introduce the following constrained optimization problem
\[
\Gamma_{\varepsilon}^{*}\in \argmin \{\err(\Gamma_F)\enspace : \enspace \text{such that}\enspace r(\Gamma_F)\leq \varepsilon \}.
\]
We make the following assumption.
\begin{assumption}
The random variable $\ent(F^{*}_X)$ has a continuous distribution.
\label{ass:nonatomic}
\end{assumption}
Assumption~\ref{ass:nonatomic} implies that the cumulative distribution function $G_{\ent}$ of $\ent(F_X^*)$ is continuous. It is required to derive a closed formula for $\varepsilon$-predictor $\Gamma^{*}_{\varepsilon}$.
\begin{proposition}
\label{prop:optimalrule}
Let $\varepsilon \in (0,1)$, and $\lambda_\varepsilon=G^{-1}_{\ent}(1-\varepsilon)$. Under Assumption~\ref{ass:nonatomic}, we have $\Gamma_{\varepsilon}^*=\Gamma_{\lambda_\varepsilon}^*$ and $r(\Gamma_\varepsilon^*)=\varepsilon$.
\end{proposition}
%The proof of this proposition relies on the method of Lagrange multipliers. 
We can use the same idea from the proof of Proposition~2 in~\citet{Zaoui2020} to demonstrate this proposition, see Supplementary material . Proposition~\ref{prop:optimalrule} defines the optimal predictor by choosing the best value for the parameter $ \lambda_{\varepsilon}$, which corresponds to the value of the generalized inverse function of the cumulative distribution $G^{-1}_{\ent}$  at $1-\varepsilon$. The advantage of this approach is that it ensures the predictor has exactly a rejection rate equal to $\varepsilon$, under Assumption~\ref{ass:nonatomic}. This precise control of the rejection rate makes the predictor both reliable and effective in the given context. In particular, the properties of the generalized inverse function~\citep{Embrechts_Hofert13} guarantee that, for $\varepsilon_2 \leq \varepsilon_1$, 
$\lambda_{\varepsilon_1} \leq \lambda_{\varepsilon_2}$, and thus 
\[
\err(\Gamma_{\varepsilon_1}^*) \leq \err(\Gamma_{\varepsilon_2}^*).
\]
In simple terms, as $\varepsilon$ increases, the error rate of $\Gamma_{\varepsilon}^*$ decreases. This shows that the method effectively reduces prediction errors as the rejection rate goes up, highlighting its usefulness and reliability in real-world applications. The performance of a predictor with reject option $\Gamma_F$ can be quantified by its excess risk
\[
\mathcal{E}_{\lambda_\varepsilon}(\Gamma_F):= \mathcal{R}_{\lambda_\varepsilon}(\Gamma_F)-\mathcal{R}_{\lambda_\varepsilon}(\Gamma_{\varepsilon}^*).
\]
The excess risk serves as a central focus of our analysis.It admits the following explicit formulation.
\begin{proposition}
Let $\varepsilon \in (0,1)$. For any predictor $\Gamma_F$, we have
\[
\mathcal{E}_{\lambda_\varepsilon}(\Gamma_F)= \mathbb{E}_X\left[\Div(F_X,F^{*}_{X})\one_{\{\Gamma_F(X)\neq \re\}}\right]
+\mathbb{E}_X\left[\big|\ent(F_{X}^{*})-\lambda_{\varepsilon}\big|\one_{\{\Gamma_F(X)\neq \Gamma_{\varepsilon}^{*}(X)\}}\right].
\]
\label{prop:excessrisk}
\end{proposition}
\vspace*{-0.5cm}
The excess risk consists of two components. The first component depends on the second Cramér's distance of $F$ when the predictor $\Gamma_F$ decides to estimate. The second component is directly linked to the behavior of the entropy function $\ent(F^{*}_X)$ around the threshold $\lambda_\varepsilon$.
\subsection{Data-driven procedure}
\label{sec:data-driven procedure}
The optimal predictor $\Gamma_\varepsilon^*$ depends on the unknown function, the conditional distribution $F^*_X$, the entropy $\ent(F^{*}_X)$ and its distribution function. Its estimation naturally relies on a semi-supervised plug-in approach. \\
\\First, we introduce two independent samples, a labeled sample $\mathcal{D}_n=\{(X_i,Y_i)\}_{i=1}^n$ and an unlabeled sample $\mathcal{D}_N=\{X_i\}_{i=1+n}^{n+N}$. Based on the first dataset $\mathcal{D}_n$, we build  an estimator $\hat{F}_{n,X}$ of $F^*_X$. According to the definition of entropy $\ent(F^*_X)$ can also be estimated using the plug-in rule and  its estimator is given as follows: \[
\widehat{\ent}(F^{*}_X)=\ent(\hat{F}_{n,X}).
\]
% in the first step, and in the second stage, we calculate the estimator of $\ent$, $\widehat{\crps}(X,Y):=\widehat{\crps}(F_{X}^*,Y)=\crps(\hat{F}_X,Y)$. 
To randomize $\widehat{\ent}$ and ensure its continuity, we introduce a random perturbation $\zeta$ distributed according to a Uniform distribution on $[0,u]$, for $u>0$ and independent from all other variables. This defines
\[
\widehat{\ent}_{\zeta}(F^{*}_X):=\ent(\hat{F}_{n,X})+\zeta.
\] 
%This step has no effect on the quality of the estimation.
Next, based on the second dataset $\mathcal{D}_N$, we calculate the empirical cumulative distribution of  $\widehat{\ent}_\zeta$
\[
\hat{G}_{\widehat{\ent}_{\zeta}}(\cdot)=\frac{1}{N}\sum_{i=1+n}^{n+N}\one_{\{\widehat{\ent}_{\zeta_i}(F^{*}_{X_i})
\leq \cdot\}}.
\]
Finally, the plug-in $\varepsilon$-predictor is given as 
\[
   \hat{\Gamma}_{\varepsilon}(X)=\begin{cases}
  \hat{F}_{n,X}  & \text{if} \;\; \hat{G}_{\widehat{\ent}_{\zeta}}(\widehat{\ent}_{\zeta}(F^{*}_X) )\leq 1-\varepsilon,\\
  \re     & \text{otherwise} \enspace .
  \end{cases}
 % \label{eq:eqOracle}
\]
%Our strategy can be elegantly summarized in Algorithm A.
It is important to highlight that the proposed methodology is flexible enough to work with any type of estimator for the conditional distribution function.
\begin{remark}
The randomization step does not affect the quality of predictions. Its role is to guarantee the continuity of $\widehat{\ent}_{\zeta}$, ensuring that the desired rejection rate is achieved. Furthermore, the consistency of $\widehat{\ent}$ ensures the consistency of $\widehat{\ent}_{\zeta}$ as $u$ tends to zero.
\end{remark}
\section{Statistical guarantees}
\label{sec:statisticalguarantees}
In this part, we provide some theoretical results justifying our approach.  In Section~\ref{sec:consitency}, we establish the consistency of the method and, subsequently in Section~\ref{sec:rates}, we derive its convergence rate when the underlying estimation of the conditional distribution is based on $k$-nearest neighbours.\\
\\ Before we proceed, we first justify that our method reaches the expected rejection rate.
\begin{proposition}
\label{prop:distributionfree}
Let $\varepsilon \in (0,1)$. Then the rejection rate of $\hat\Gamma_\varepsilon$ satisfies
    \[
    \mathbb{E}\left[|r(\hat{\Gamma}_{\varepsilon})-\varepsilon|\right]\leq C N^{-1/2},
    \]
    where $C>0$ is an absolute constant.
\end{proposition}
This states that the rejection rate is well-controlled up to a term of order $N^{-1/2}$, depending on the size of the second unlabeled sample. We highlight that the result is distribution-free, i.e. it does not depend on the distribution of $(X,Y)$. This aligns with the same rate of convergence observed within the framework of both regression and classification with reject option~\citep{denis_hebiri_2020, Zaoui2020}.
\subsection{Consistency result}
\label{sec:consitency}
 In this part, we establish the consistency of the plug-in $\varepsilon$-estimator based on the estimator $\hat{F}_X$ and $\widehat{\ent}$.  It is worth noting that $\widehat{\ent}$ itself is derived from the estimator $\hat{F}_X$. We describe a general bound of the excess risk of $\hat{\Gamma}_{\varepsilon}$. We make the following assumption.
 \begin{assumption}
 There exists a constant $M\geq 0$ such that 
 \[
 \ent(F^{*}_x) \leq M  \enspace \textit{for all} \enspace x \in \mathbb{R}^d.
 \]
 \label{ass:boundedentropy}
 \end{assumption}
\begin{theorem}
\label{thm:consistency}
Let $\varepsilon\in (0,1)$. Under Assumption~\ref{ass:nonatomic}-~\ref{ass:boundedentropy}, it holds that 
\[
\mathbb{E}\left[\mathcal{E}_{\lambda_\varepsilon}(\hat{\Gamma}_{\varepsilon})\right]
\leq 2\ \mathbb{E}\left[\Div(\hat{F}_{n,X},F^{*}_{X})\right] + \mathbb{E}\left[\big|\ent(\hat F_{n,X})-\ent(F^{*}_X)\big|\right]+\frac{MC}{\sqrt{N}}+u.
\]
\end{theorem}
This Theorem provides an upper bound for the excess risk of $\hat{\Gamma}_{\varepsilon}$. This bound depends on the second Cramér's distance for estimating the conditional distribution function, the $\ell_1$-norm of $\widehat{\ent}$, a bound that controls the rejection rate of $\hat{\Gamma}_{\varepsilon}$, and the parameter associated with the randomization. % It is worth noting that this bound does not depend on the choice of the estimator $\widehat{\ent}$. %This is because the $\ell_1$-norm of $\widehat{\ent}$ can be controlled through the Wasserstein distance of order $1$, which serves as an elegant measure of discrepancies between distributions. In particular, we can establish the following result.
Observe that the divergence term is bounded by the Wasserstein distance of order $1$, and the $\ell_1$-norm of $\widehat{\ent}$ can similarly be controlled by this distance (see Lemma~\ref{lem:boundedentropy}), which offers a graceful measure of discrepancies between distributions. This leads us to the following result.
\begin{cor}\label{cor:bound-Wasserstein}
Let $\varepsilon\in (0,1)$. Under Assumption~\ref{ass:nonatomic}-~\ref{ass:boundedentropy}, it holds that 
\[
\mathbb{E}\left[\mathcal{E}_{\lambda_\varepsilon}(\hat{\Gamma}_{\varepsilon})\right]
\leq 3\mathbb{E}\left[W_{1}(\hat{F}_{n,X},F_{X}^{*})\big|\right]+\frac{MC}{\sqrt{N}}+u.
\]
\end{cor}
In particular, we can establish the following result.
\begin{cor}
\label{cor:consistencyexcessrik}
Let $\varepsilon \in (0,1)$. Assume that $u=u_n \to 0$ and $\mathbb{E}[W_{1}(\hat{F}_{n,X},F_{X}^{*})] \to 0$ as $n\to \infty$, then
\[
\mathbb{E}\left[\mathcal{E}_{\lambda_\varepsilon}(\hat{\Gamma}_{\varepsilon})\right]\to 0 \enspace \text{as}\enspace n,N \to \infty.
\]
\end{cor}

Proposition~\ref{prop:Ivarepsilon} and Corollary~\ref{cor:consistencyexcessrik} together establish that $\hat{\Gamma}_{\varepsilon}$ performs asymptotically as well as $\Gamma_{\varepsilon}^{*}$, both in terms of error and rejection rate, provided that the estimator of $\hat{F}_{n,X}$ is consistent with respect to the Wasserstein distance.
%%%%%%%%%%%%%%%%%%%%%%%%%%%%%%%%%%%%%%%%%%%%%%%%%%%%%%%%%%%%%%%%%%%%%%%%%%%%%%%%%%%%%%%%%%%%%%%%%%%%%%%%%%%%%%%
\subsection{Rates of convergence for distributional KNN}
\label{sec:rates}
In this section, we apply our approach to the distributional $k$-nearest neighbors (KNN) for both the estimation of the conditional distribution $F^*_x$ and the cumulative distribution of the entropy function $\ent(F^*_X)$.\\
 The $k$-nearest neighbor ($k$-NN) method is a well-known and widely used tool in regression and classification. For any $x\in \mathbb{R}^d$, we denote by $(X_{(i,n)}(x), Y_{(i,n)}(x))$, $i=1,\ldots,n$, the reordered data according to the Euclidean distance in $\mathbb{R}^d$, meaning that 
 \[
 \|X_{(i,n)}(x)-x\|\leq \|X_{(j,n)}(x)-x\|
 \]
for all $i<j$ in $\{1,\ldots,n\}$. Here, $X_{(i,n)}(x)$ is the $i$-th nearest neighbor of $x$ and $Y_{(i,n)}(x))$ is the $Y_i$ corresponding to $X_{(i,n)}(x)$. In case of distance ties, we declare $X_{(i,n)}(x)$ is considered to be closer to $x$. The estimator of the k-nearest neighbors of the regression function is therefore written as
\[
\hat{f}_n(x)=\frac{1}{k}\sum_{i=1}^{k}Y_{(i,n)}(x)=\sum_{i=1}^{n}w_{ni}(x)Y_i,
\] 
where the weights $w_{ni}(x)$, $i=1,\ldots,n$, are non negative, sum to one, and are defined as
\[
w_{ni}(x) =\begin{cases}
  \frac{1}{k}  & \text{if $X_i$ is one of the $k$ nearest neighbours of $x$}, \\
  0     & \text{otherwise}.
  \end{cases}
\]
The statistical properties of this estimator are extensively discussed in the following references books~\citep{Gyofri_Kohler_Krzyzak_Walk02, Biau_Devroye15}.
 Within the framework of distributional regression, this method can be easily adapted to construct the predictive distribution. By extending the basic principles of $k$-NN regression, we define the predictive distribution as:
\[
\hat F_{n,x}(y)=\frac{1}{k}\sum_{i=1}^{k} \mathds{1}_{\{Y_{(i,n)}(x)\leq y\}}=\sum_{i=1}^{n} w_{ni}(x)\mathds{1}_{\{Y_i\leq y\}}.
\]
This method, referred to as the Analog method, has been widely used in the field of statistical post-processing for weather forecasts \citep{ZT89}. Convergence rate with respect to the Wasserstein distance  of order $1$ have been recently established in \citet{Dombry_24} under the following following assumption.
%\begin{assumption}[Strong density assumption]
%\label{ass:StrongDensityAssumption}
%The marginal distribution $\mathbb{P}_{X}$ satisfies the strong density assumption
%\begin{itemize}
%\item[$\bullet$] $\mathbb{P}_{X}$ is supported on a compact regular set $\mathcal{C}\subset \mathbb{R}^d $,
%\item[$\bullet$]$\mathbb{P}_X$ admits a density $\mu$ w.r.t. to the Lebesgue measure such that  $0 < \mu_{\min} \leq \mu(x) \leq \mu_{\max}<\infty$, for all $x \in \mathcal{C}$.
%\end{itemize}
%\end{assumption}

%This assumption ensures that the entropy function $\ent$ is finite.
\begin{assumption} There exists a compact $\mathcal{X}\in\mathbb{R}^d$ and constants  $ M,L>0$, $h\in (0,1]$ such that:
\begin{enumerate}[label=\textnormal{(\roman*)}]
\item \label{ass3:i} $\mathbb{E}[|Y|]<\infty$ and $\mathbb{P}(X\in\mathcal{X})=1$,
\item \label{ass3:ii} $W_{1}(F_{x}^*,F_{x'}^*)\leq L\|x-x'\|^h$ for all $x,x' \in \mathcal{X}$,
\item \label{ass3:iii} $\int_{\mathbb{R}}\sqrt{F^{*}_x(y)(1-F^{*}_x(y))}dy\leq M \enspace \textit{for all} \enspace x\in \mathcal{X}$.
\end{enumerate}  
\label{ass:cramer_wasserstein}
\end{assumption}
 Because $z\leq \sqrt{z}$ for $z\in[0,1]$, it is straightforward to see that Assumption~\ref{ass:cramer_wasserstein}~\ref{ass3:iii} implies Assumption~\ref{ass:boundedentropy}.

\begin{theorem}{\citep{Dombry_24}}
Under Assumption~\ref{ass:cramer_wasserstein},  there exist  a constant $c_d>0$ depending on the dimension $d$ only, such that 
\[\mathbb{E}\left[W_1(\hat F_{n,X},F^{*}_{X})\right]\leq \begin{cases} L8^{h/2} (k/n)^{h/2} +Mk^{-1/2}& \mbox{if $d=1$}, \\
Lc_d^{h/2} (k/n)^{h/d} +Mk^{-1/2}
& \mbox{if $d\geq 2$}. \end{cases}
\]
 \label{thm:ratesW1}
\end{theorem}
The bound is non asymptotic and holds for all $n\geq 1$ and $1\leq k\leq n$. The usual conditions $k\to\infty$, $k/n\to 0$ ensure consistency. Furthermore, in dimension $d\geq 2$, it is known that optimizing with respect to $k$ leads to a minimax rate of convergence of order $n^{-h/(2h+d)}$, see Theorem~3 in \citet{Dombry_24}. 

Recall that for the distributional regression with reject option, the  excess risk of our approach is bounded thanks to the Wasserstein distance. Thus, applying Corollary~\ref{cor:bound-Wasserstein} and Theorem~\ref{thm:ratesW1}, we obtain the following result.
\begin{cor}
Suppose Assumptions~\ref{ass:nonatomic} and~\ref{ass:cramer_wasserstein} are satisfied and let $\varepsilon \in (0,1)$ and $u=n^{-1}$. Then, %for $d=1$
%\[
%\mathbb{E}\left[\mathcal{E}_{\lambda_\varepsilon}(\hat{\Gamma}_{\varepsilon})\right]\leq  3L8^{h/2} (k/%n)^{h/2} +3Mk^{-1/2}+MCN^{-1/2}+n^{-1}
%\]
%and, for $d\geq 2$,
%\[
%\mathbb{E}\left[\mathcal{E}_{\lambda_\varepsilon}(\hat{\Gamma}_{\varepsilon})\right]\leq 3Lc_d^{h/2} (k/n)^{h/d} +3Mk^{-1/2}+MCN^{-1/2}+n^{-1}.
%\]
\[
\mathbb{E}\left[\mathcal{E}_{\lambda_\varepsilon}(\hat{\Gamma}_{\varepsilon})\right]\leq\begin{cases}
  3L8^{h/2} (k/n)^{h/2} +3Mk^{-1/2}+MCN^{-1/2}+n^{-1} & \text{if} \enspace d=1,\\
 3Lc_d^{h/2} (k/n)^{h/d} +3Mk^{-1/2}+MCN^{-1/2}+n^{-1}    & \text{if} \enspace d\geq 2.
  \end{cases}
\]
\end{cor}
Again, this upper bound for the expected excess risk is non asymptotic and holds for all $n\geq 1$ and $1\leq k\leq n$. The first two terms are bias and variance terms from the estimation of the conditional distribution, the third term is due to the calibration of the procedure when estimating the entropy function  and the fourth term to the randomization with $u=n^{-1}$. 

Optimizing with respect to $k$, we obtain the following rate of convergence for the oracle with $k\asymp n^{\frac{h}{h+1}}$ if $d=1$ and $k\asymp n^{\frac{2h}{2h+d}}$ if $d\geq 2$:
\[
\mathbb{E}\left[\mathcal{E}_{\lambda_\varepsilon}(\hat{\Gamma}_{\varepsilon})\right]\lesssim\begin{cases}
  n^{-\frac{h}{2h+2}}+N^{-1/2}  & \text{if} \enspace d=1,\\
 n^{-\frac{h}{2h+d}}+N^{-1/2}    & \text{if} \enspace d\geq 2.
  \end{cases}
\]
We can deduce that if the number of unlabelled data $N$ is sufficiently large, then this rate becomes identical to the rate of convergence of $\hat F_{n,X}$ with respect to the Wasserstein distance. In the case where $N<n$, and $n^{-\frac{h}{2h+2}}$ (when $d=1$) or $n^{-\frac{h}{2h+d}}$  (when $d>2$) goes faster to zero, this term remains limiting. However, it is reasonable to think that, in our situation, having enough data will make this term negligible compared to the others.
%\begin{definition} Let $h\in (0,1]$, $C>0$ and $M>0$. Define $\mathcal{D}^{(h,C,M)}$ as the class of distributions $\mathbb{P}$ such that $F_{x}^*$ satisfies:
%\begin{itemize}
%\item[i)]$X\in [0,1]^d$ $\mathbb{P}_X$-a.s.;
%\item[ii)] for all $x\in [0,1]^d$, $\ent(F_{x}^*)\leq M$;
%\item[iii)]$\|F_{x}^*-F_{x'}^*\|_{L^2}\leq \|x-x'\|^h$ for all $x,x' \in [0,1]^d$.
%\end{itemize}
%\end{definition}
%\begin{definition}
%Define $\mathcal{D}^{(h,C,M)}$ as the class of distributions $\mathbb{P}$ that satisfies:
%\begin{itemize}
%\item[i)]$X\in [0,1]^d$ $\mathbb{P}_X$-a.s. and $\mathbb{E}[|Y|]<\infty$;
%\item[ii)] for all $x\in [0,1]^d$, $\int_{\mathbb{R}}\sqrt{F_{x}^*(u)(1-F_{x}^*(u)}du\leq M$;
%\item[iii)]$W_{1}(F_{x}^*,F_{x'}^*)\leq \|x-x'\|^H$ for all $x,x' \in [0,1]^d$.
%\end{itemize}
%\end{definition}
%%%%%%%%%%%%%%%%%%%%%%%%%%%%%%%%%%%%%%%%%%%%%%%%%%%%%%%%%%%%%%%%%%%%%%%%%%%%%%%%%%%%%%%%%%%%%%%%%%%%%%%%%%%%%%%%%%%%%%%%%%%%%%%%%%%%%%%%%%%%%%%%%%%%%%%%%%%%%

\section{Numerical results}
\label{sec:numerical}
This section examines the numerical performance of the proposed plug-in $\varepsilon$-predictor through a simulation study. The source code for these experiments is publicly available at \url{https://github.com/ZaouiAmed/DistributionalRegression_RejectOption.}\\% \url{link_to_github_account} (will be provided later for the sake of anonymity).\\
\\{\bf{Datasets.}} The performance is evaluated on three benchmark datasets: QSAR Aquatic Toxicity, Airfoil Self-Noise, and Concrete Compressive Strength, all sourced from the UCI Machine Learning Repository. A brief description of these datasets is provided below.
\begin{itemize}
\item The \emph{QSAR Aquatic Toxicity}, referred to as \texttt{qsar}, consists of $546$ samples with $8$ numerical features used to predict acute toxicity in Pimephales promelas. The toxicity output ranges from $0.12$ to $10.05$, with evidence of low heteroscedasticity.
\item The \emph{Airfoil Self-Noise dataset}, referred to as \texttt{airfoil}, contains 1503 observations with 5 features measured from aerodynamic and acoustic tests. The output, representing the scaled sound pressure level in decibels, ranges from 103 to 140 and exhibits strong heteroscedasticity.
\item The \emph{Concrete compressive strength}, referred to as \texttt{concrete}, contains $1030$ observations with $8$ features used to predict The concrete compressive strength. This target variable, which represents the maximum compressive strength the concrete can withstand, ranges from $2.33$ to $82.6$ megapascals (MPa). The dataset demonstrates strong heteroscedasticity.
\end{itemize}
The first two datasets employed in the numerical experiments within the context of heteroscedastic regression with reject option~\citep{Zaoui2020} and distributional regression~\citep{Dombry_Zaoui24}.\\
\\{\textbf{Models.}} Recall that the plug-in predictor depends on estimators for the conditional function and the entropy function. The entropy function estimator is derived using the plug-in rule and relies on the conditional distribution function estimator. We construct two plug-in $\varepsilon$-predictors, each based on  distributional random forests (DRF) and distributional k-nearest neighbors (KNN) models, respectively. We provide a brief description of DRF.
\\ \underline{Distributional Random Forests (DRF).} Random Forest \citep{Breiman_2001} is a powerful machine learning algorithm in the supervised learning category, suitable for both classification and regression tasks. It works by creating a collection of decision trees, each trained on different random subsets of the data. The final prediction is made by aggregating the results from all the trees, which helps improve accuracy and reduce overfitting. \citet{Cevid_MichelNBM22} proposed Random Forests for distributional regression. Let $T_1,\ldots,T_B$ be the randomized regression trees built on bootstrap samples of the original data. By applying the weights from these Random Forests to standard regression, the predictive distribution is derived as
\begin{eqnarray*}
\hat F_{n,x}(y)=\sum_{i=1}^n \left(\frac{1}{B}\sum_{b=1}^B \frac{\mathds{1}_{\{X_i\in L_b(x)\}}}{|L_b(x)|}\right)\mathds{1}_{\{Y_i\leq y\}}
=\sum_{i=1}^n w_{ni}(x)\mathds{1}_{\{Y_i\leq y\}},
\end{eqnarray*}
where $L_b(x)$ denotes the leaf containing $x$ in the tree $T_b$. The weights $w_{ni}(x)$ are positive and add up to $1$.
We apply the implementations of these methods from the R packages \texttt{KernelKnn} and \texttt{DRF}. For DRF, preliminary analysis suggests that the most reasonable choices for the key parameters include \texttt{num.trees} $= 1000$, \texttt{sample.fraction} $= 0.9$, \texttt{min.node.size} $=1$, with the default values for the other parameters. The parameters $\widehat{\texttt{k}}$ and $\widehat{\texttt{mtry}}$ are selected via model selection (see \citet{Dombry_Zaoui24} for more details). Finally, the R package \texttt{ScoringRule} is used to compute the CRPS.\\
\\{\bf{Methodology.}} For all datasets, the data is divided into three subsets: $50\%$ labeled training data, $20\%$ unlabeled training data, and $30\%$ test data.  
\begin{itemize}
    \item The labeled training set is used to calculate the estimators $\hat F_{n,x}$ using KNN and DRF.
    \item The unlabeled training set is used to compute the empirical cumulative distribution function of $\widehat{\ent}_{\zeta}$, with a random perturbation $\zeta \sim \mathcal{U}([0, 10^{-10}])$ . The explicit formula for $\widehat{\ent}$ for the KNN and DRF models is provided in Lemma~\ref{lem:expliciteformulaENT} (Appendix).
\end{itemize}
For each $\varepsilon \in \{i/10 : i = 0, \ldots, 9\}$ and each plug-in $\varepsilon$-predictor, we calculate the empirical rejection rate ($\hat{r}$) and the empirical error ($\widehat{\text{Err}}$).
%we evaluate the following quantities using the test data:
%\begin{enumerate}
 %   \item The \textbf{empirical rejection rate} ($\hat{r}$),
  %  \item The \textbf{empirical error} ($\widehat{\text{Err}}$).
%\end{enumerate}
 This process is repeated 100 times, and we report the average performance along with its standard deviation on the test data.\\
 \\ \textbf{Results.} The results obtained from our experiments are summarized in Table~\ref{result1} and Figure~\ref{fig:result}. We make the following observations. First, the empirical errors of the plug-in $\varepsilon$-predictors decrease with respect to $\varepsilon$ for both datasets, see Figure~\ref{fig:result}. This indicates that the rejection option contributes to improving the risk. This shows that the method works well. %For the Concrete and Airfoil datasets, the regression distribution problem is quite challenging, which makes the decrease in error slower and slightly less significant.
 Additionally, the empirical rejection rates are very close to the expected values, confirming that the approach is accurate, see Table~\ref{result1}. These results strongly support our theory and give us confidence that the method behaves as expected. Second, it is important to emphasize that a small unlabeled dataset allows for effective control of the rejection rate.  Finally, we observe that the approach based on DRF is always better than the KNN-based approach. We conclude that better calibration (rather than estimation) of the thresholding of the entropy for CRPS leads to a more effective methodology.

\begin{table}[!ht]
\caption{Performances of the three plug-in $\varepsilon$-predictors on the real datasets $\texttt{qsar}$, $\texttt{concrete}$, and $\texttt{airfoil}$.}
\centering
{\tiny
{\setlength{\tabcolsep}{9pt}
\vspace*{0.25cm}
\begin{tabular}{||l||cc|cc||cc|cc||}
\hline
\multicolumn{1}{c}{} & \multicolumn{4}{c}{\texttt{qsar}} & \multicolumn{4}{c}{\texttt{concrete}} \\ \hline
\multicolumn{1}{c}{} & \multicolumn{2}{c}{\texttt{DRF}} & \multicolumn{2}{c}{\texttt{KNN}} & \multicolumn{2}{c}{\texttt{DRF}} & \multicolumn{2}{c}{\texttt{KNN}} \\ \hline
\noalign{\smallskip}
\\ \hline \noalign{\smallskip}
$\varepsilon$ & $\widehat{\err}$ & $\hat{r}$ & $\widehat{\err}$ & $\hat{r}$ & $\widehat{\err}$ & $\hat{r}$ & $\widehat{\err}$ & $\hat{r}$  \\ \hline \noalign{\smallskip}
$0$   & 0.65 (0.04) &  0.00 (0.00) & 0.85 (0.05) & 0.00 (0.00)  &  3.55 (0.14)  & 0.00 (0.00)  & 5.80 (0.25) & 0.00 (0.00)\\
$0.1$ & 0.59 (0.04) & 0.10 (0.03)  & 0.81 (0.06) & 0.10 (0.04) &  3.23 (0.14) & 0.10 (0.03) & 5.38 (0.35) & 0.10 (0.03) \\
$0.2$ & 0.55 (0.03) &  0.20 (0.05) & 0.79 (0.06) & 0.20 (0.05) &  3.04 (0.14) & 0.21 (0.03) & 5.17 (0.35) & 0.20 (0.04) \\
$0.3$ & 0.52 (0.04) & 0.29 (0.06) & 0.75 (0.07) & 0.30 (0.07) & 2.88 (0.14) & 0.30 (0.04) & 4.98 (0.36) & 0.30 (0.05)  \\
$0.4$ & 0.50 (0.04) & 0.39 (0.06)  & 0.73 (0.09) & 0.40 (0.06) &  2.75 (0.15) & 0.40 (0.04) & 4.66 (0.39) & 0.40 (0.05)  \\
$0.5$ & 0.46 (0.04) & 0.50 (0.07) & 0.69 (0.10) & 0.51 (0.07)  &  2.61 (0.14) & 0.49 (0.05) & 4.56 (0.43) & 0.50 (0.05) \\
$0.6$ & 0.44 (0.04) & 0.60 (0.06)  & 0.63 (0.10) & 0.60 (0.06)  & 2.47 (0.15) & 0.59 (0.04) & 4.46 (0.45) & 0.60 (0.04)  \\
$0.7$ & 0.41 (0.05) & 0.70 (0.05)  & 0.57 (0.12) & 0.70 (0.05)  &  2.29 (0.15) & 0.70 (0.04) & 4.16 (0.47) & 0.70 (0.04)\\
$0.8$ & 0.39 (0.06) &  0.81 (0.04) & 0.51 (0.15) & 0.81 (0.05) &  2.12 (0.21)& 0.80 (0.03)  & 3.93 (0.59) & 0.80 (0.04)\\
$0.9$ & 0.37 (0.09) &  0.91 (0.03) & 0.44 (0.15) & 0.90 (0.04) & 1.92 (0.22) & 0.90 (0.03) & 3.63 (0.83) & 0.90 (0.03) \\
\end{tabular}}
}

\quad
{\tiny
{\setlength{\tabcolsep}{9pt}
\vspace*{0.25cm}
\begin{tabular}{||l || cc | cc || }
\multicolumn{1}{c}{} &  \multicolumn{4}{c}{{\texttt{airfoil}}}\\ \hline
\multicolumn{1}{c}{}  &  \multicolumn{2}{c}{{\texttt{DRF}}} & \multicolumn{2}{c}{{\texttt{KNN}}} 
\\ \hline \noalign{\smallskip}
$\varepsilon$ & $\widehat{\err}$ & $\hat{r}$ & $\widehat{\err}$ & $\hat{r}$ \\ \hline \noalign{\smallskip}
$0$   &  1.53 (0.05) & 0.00 (0.00)  & 3.41 (0.11) & 0.00 (0.00) \\
$0.1$ & 1.41 (0.05) & 0.10 (0.02) & 3.29 (0.11) & 0.10 (0.02) \\
$0.2$ &  1.34 (0.05) &  0.20 (0.03) & 3.18 (0.13) & 0.21 (0.03) \\
$0.3$ &  1.26 (0.06) &  0.30 (0.03)&  3.05 (0.13) & 0.30 (0.03) \\
$0.4$ &  1.19 (0.05) & 0.40 (0.04) & 2.85 (0.13) & 0.40 (0.06) \\
$0.5$ &  1.11 (0.06)  & 0.50 (0.04) & 2.73 (0.14) & 0.50 (0.04) \\
$0.6$ &  1.03 (0.05) & 0.60 (0.04) & 2.54 (0.15) & 0.60 (0.04) \\
$0.7$ &  0.97 (0.06) & 0.70 (0.03) & 2.43 (0.18) & 0.70 (0.03) \\
$0.8$ &  0.91 (0.07) & 0.79 (0.03) & 2.29 (0.24) & 0.80 (0.03) \\
$0.9$ &  0.78 (0.09) & 0.90 (0.02) & 2.10 (0.32) & 0.90 (0.02) \\
\end{tabular}}
}
\label{result1}
\end{table}

\begin{figure}[!ht]
\begin{center}
\begin{tabular}{ccc}
\includegraphics[height = 0.33 \columnwidth]{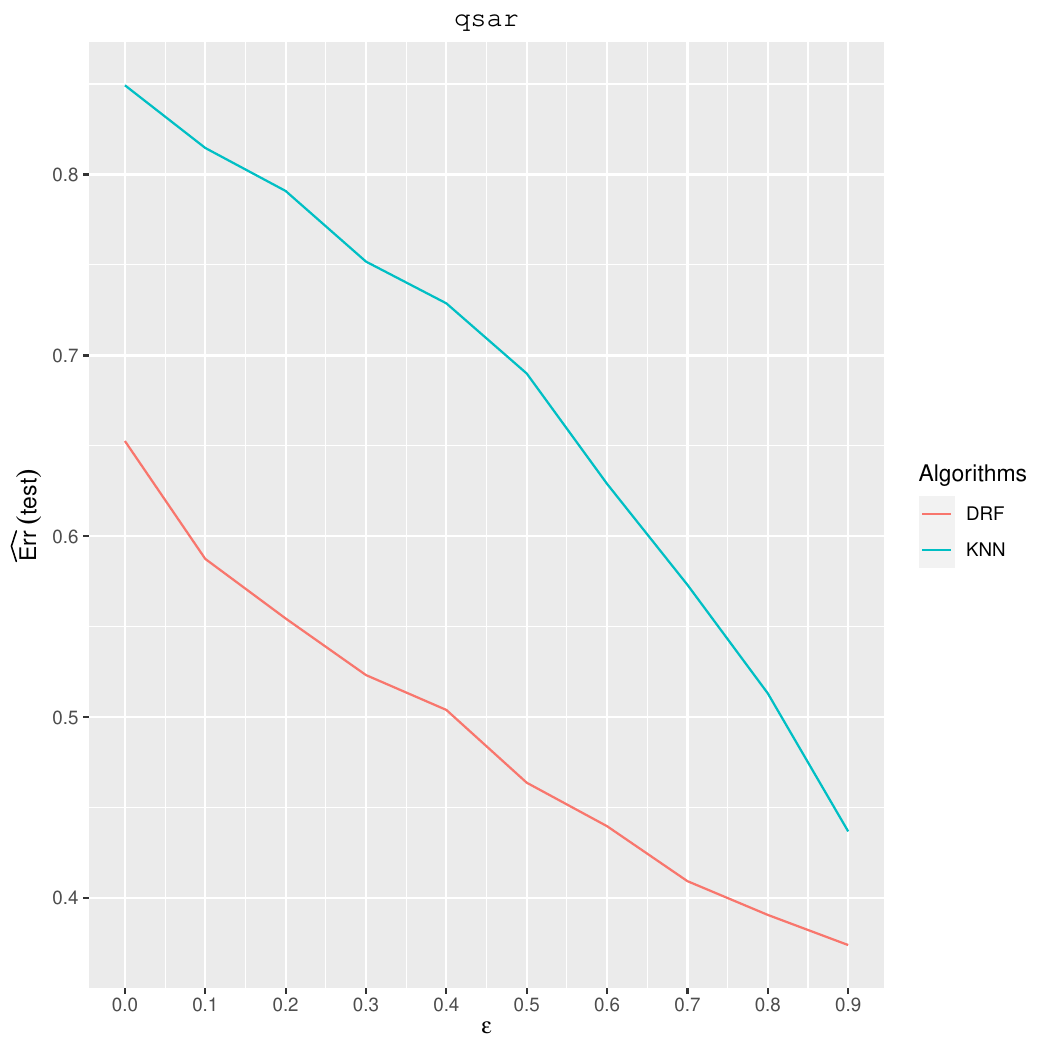}  &
\includegraphics[height = 0.33 \columnwidth]{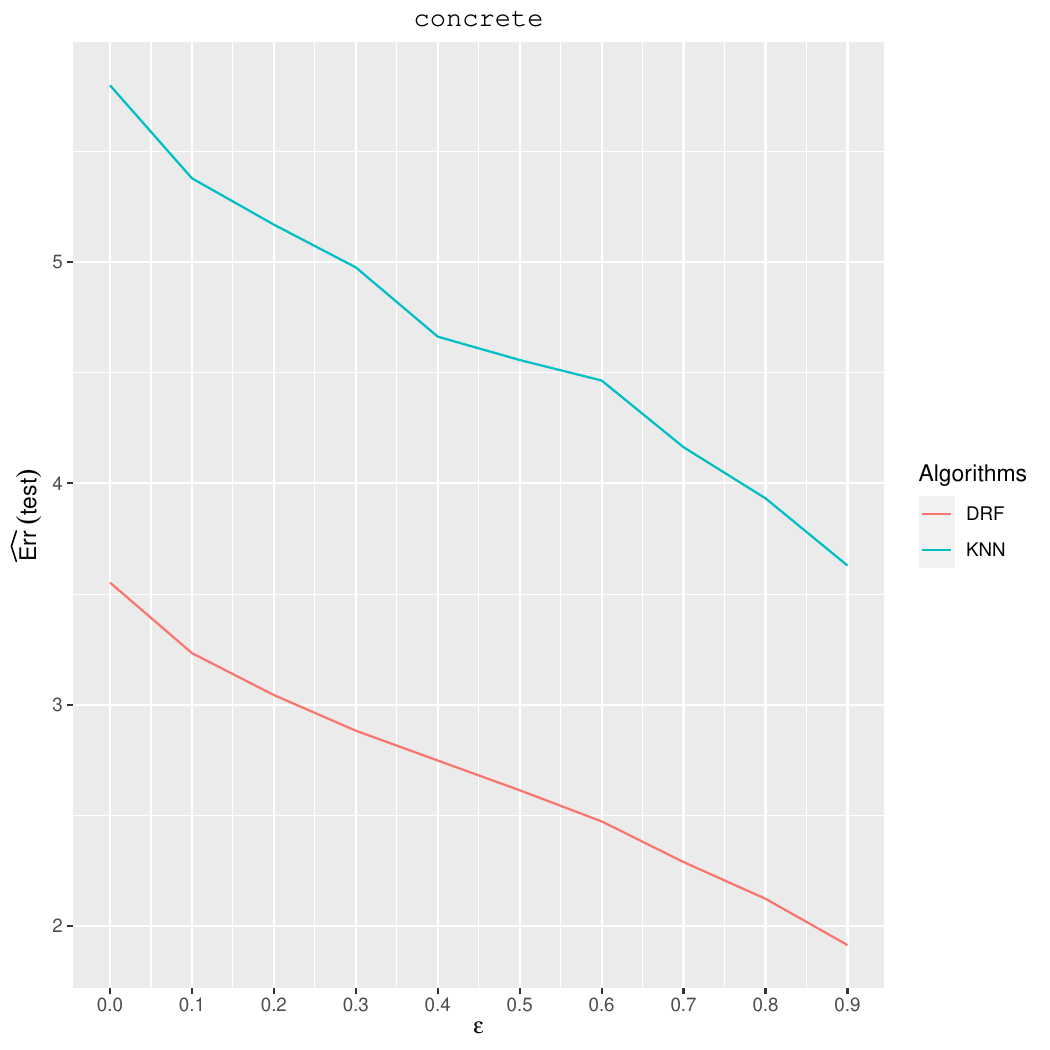}  &
\includegraphics[height = 0.33 \columnwidth]{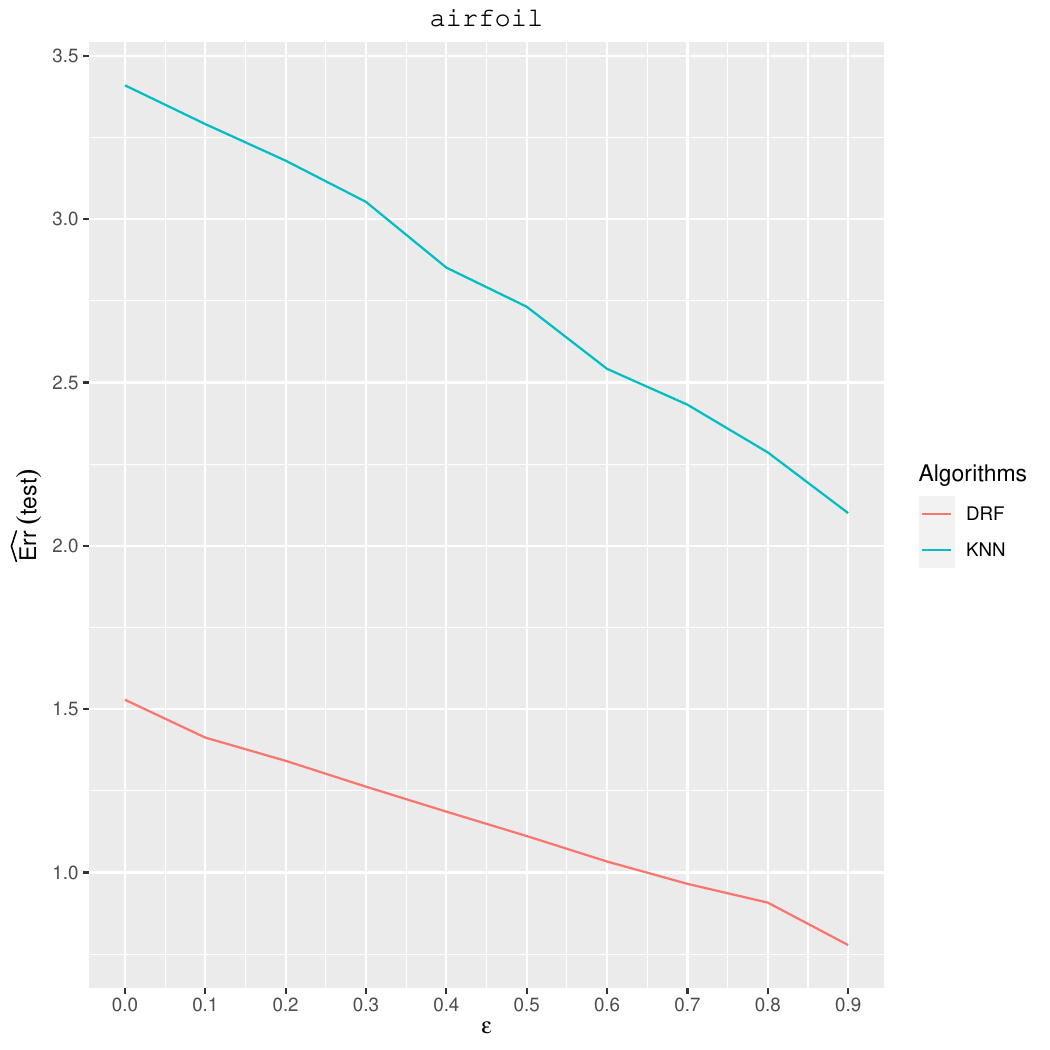} \\
\end{tabular}
\end{center}
\caption{\label{fig:result}
Visual description of the performance of two plug-in $\varepsilon$-predictors on the $\texttt{qsar}$, \texttt{concrete}, and $\texttt{airfoil}$ datasets.}
\end{figure}
\section{Conclusion and perspectives}
In this work, we have studied the use of the reject option in the distributional regression setting. When the risk is measured with the CRPS and the rejection rate is fixed, the optimal rule was given and we have shown in particular that the reject option must be used when the entropy function of the CRPS exceeds some threshold. Using the plug-in principle, we have proposed  to estimate the optimal rule with an approach applicable to any estimator of the conditional distribution function. General consistency results for the rejection rate and excess risk have been derived. Additionally, we established more refined rates of convergence when the conditional distribution function is estimated using the ditributional $k$-nearest neighbour method.  Our future work will focus on studying the problem in high-dimensional settings. The objective will be to adapt and extend our approach to ensure its effectiveness and robustness in more complex and large-scale contexts.
%\\ We acknowledge that estimation performance may be influenced by several factors, including sample size and data dimensionality. A more in-depth analysis of these effects would help better understand the limitations and optimal application conditions of our method. Therefore, our future work will focus on studying the problem in high-dimensional settings. The objective will be to adapt and extend our approach to ensure its effectiveness and robustness in more complex and large-scale contexts.
\section{Acknowledgement}
This work was supported by the Chrysalide Nouveaux Arrivants project, LmB, UMLP.
%\newpage
\bibliographystyle{apalike}
\bibliography{sample.bib}

\newpage
\appendix
\section*{Supplementary material}

This section gathers the proof of our results.
%%%%%%%%%%%%%%%%%%%%%%%%%%%%%%%%%%%%%%%%%%%%%%%%%%%%%%%%%%%%%%%%%%%%%%%%%%%%%%%%%%%%%%%%%%%%%%%%%%%%%%%%%%%%%%%%%%%%%%%%%%%%%%%%%%%%%%%%%%%%%%%%%%%%%%%%%%%%%%%%%%%%%%%%%%%%%%%%%%%%%%%%%%%%%%%%%%%%%%%%%%%%%%%%%%%%%%%%%%%%%%%%

%%%%%%%%%%%%%%%%%%%%%%%%%%%%%%%%%%%%%%%%%%%%%%%%%%%%%%%%%%%%%%%%%%%%%%%%%%%%%%%%%%%%%%%%%%%%%%%%%%%%%%%%%%%%%%%%%%%%%%%%%%%%%%%%%%%%%%%
%%%%%%%%%%%%%%%%%%%%%%%%%%%%%%%%%%%%%%%%%%%%%
\section{Technical lemmas}
In this section we gather several technical results which are used to derive the contributions of this work.
\begin{lemma}
Let $F_X^*$ be the conditional distribution of $Y$ given $X$. Then, we have
\[
\mathbb{E}\left[\crps(F_{X}^{*},Y)|X\right]=\ent(F_X^*).
\]
\label{lem:entropy}
\end{lemma}
\begin{proof}
By the definition of $F^*_X$, we have $F^{*}_X(y)=\mathbb{P}(Y\leq y|X)=\mathbb{E}[\one_{\{Y\leq y\}}|X]$. Using Fubini's Theorem, we get
\begin{eqnarray*}
\mathbb{E}\left[\crps(F_{X}^{*},Y)|X\right]&=&\int_{\mathbb{R}}\mathbb{E}\left[(F_{X}^{*}(u))^2-2\one_{\{Y\leq u\}}F_{X}^{*}(u)+\one_{\{Y\leq u\}}|X\right]du\\
&=& \int_{\mathbb{R}}\left((F_{X}^{*}(u))^2-2\mathbb{E}\left[\one_{\{Y\leq u\}}|X\right]F_{X}^{*}(u)+\mathbb{E}\left[\one_{\{Y\leq u\}}|X\right]\right)du\\
&=&\int_{\mathbb{R}}F_{X}^{*}(u)(1-F_{X}^{*}(u))du=\ent(F^{*}_X).
\end{eqnarray*}
\end{proof}
\begin{lemma}[Dvoretzky-Kiefer-Wolfowitz Inequality~\citep{Massart90}] Let $Z_1,\ldots,Z_N$ be i.i.d. real-valued random variables with distribution function $H_Z$ and let $\hat{H}_Z$ be the empirical distribution function. Then for every $\beta>0$ we have
\[
\mathbb{P}\left(\sup_{z\in \mathbb{R}}\big|\hat{H}_Z(z)-H_{Z}(z)\big|\geq \beta\right)\leq 2 \exp(-2N \beta^2).
\]
\label{lem:DKW}
\end{lemma}
\begin{lemma}
Let $H$ and $K$ be two distribution functions. Then, the following inequalities hold
\[
\big|\ent(H)-\ent(K)\big|\leq W_{1}(H,K), \enspace \text{and}\enspace \Div(H,K) \leq W_{1}(H,K).
\]
\label{lem:boundedentropy}
\end{lemma}
\begin{proof}
By the definition of entropy $\ent$, we have 
\begin{eqnarray*}
\big|\ent(H)-\ent(K)\big| &\leq & \int_{\mathbb{R}}\big|H(u)(1-H(u))-K(u)(1-K(u)\big|du \\
&=& \int_{\mathbb{R}}\big|H(u)-K(u)+K^2(u)-H^2(u)\big|du\\
&\leq &\int_{\mathbb{R}}\big|H(u)-K(u)\big|\big|1-K(u)-H(u)\big|du\\
&\leq &\int_{\mathbb{R}}\big|H(u)-K(u)\big|du=W_{1}(H,K). \\
\end{eqnarray*}
The second inequality is obvious.
\end{proof}
\begin{lemma}[Lemma~1 in~\citep{denis_hebiri_2020}]
\label{lem:DH20}
Let $X$ be a real random variable, $(X_n)_{n\geq 1}$ be a sequence of real random variables and $t_0 \in \mathbb{R}$.
Assume that there exist $C_1 > 0 $ and $\gamma_0>0$ such that
$$
\mathbb{P}_X \left( | X - t_0| \leq \delta  \right) \leq C_1 \delta^{\gamma_0} , \quad \forall \delta > 0\enspace,
$$
and a sequence of positive numbers $a_n$ tends towards infinity, $C_2$, $C_3$  some positive constants such that
$$
\mathbb{P}_{X_{n}} \left( | X_n - X| \geq \delta|X  \right) \leq C_2 \exp\left(-C_{3}a_{n}\delta^2\right) , \quad \forall \delta > 0,  \quad \forall n\in \mathbb{N}.
$$
Then, there exists $C>0$ depending only on $C_1,C_2$ and $C_3$, such that
\begin{equation*}
 |\mathbb{E}\left[\one_{X_n\geq t_0}-\one_{X\geq t_0}\right]|
 %&\leq & \mathbb{E}\left[|\one_{X_n\geq t_0}-\one_{X\geq t_0}|\right]\\
 %&\leq & \mathbb{P}\left(|X_n-X|\geq |X-t_0|\right)\\
 % &\leq & %
 \leq C a_{n}^{-\gamma_0/2}.
\end{equation*}
\end{lemma}
\begin{lemma}
\label{lem:expectedCrps}
Let $H$ and $K$ be two distribution functions with a finite absolute moment. Then,
\[
\overline{\crps}(H,K)=\mathbb{E}_{Y\sim K}\left[\crps(H,Y)\right]=\ent(K)+\int_{\mathbb{R}}(H(u)-K(u))^2 du
\]
\end{lemma}
\begin{proof}
Recall that for $u\in \mathbb{R}$, $K(u)=\mathbb{E}_{Y\sim K}\left[\one_{Y\leq u}\right]$. We use the definition of the $\crps$ and apply Fubini's theorem to obtain
\begin{eqnarray*}
\overline{\crps}(H,K)&=&\mathbb{E}_{Y\sim K}\left[\crps(H,Y)\right]\\
 &=& \int_{\mathbb{R}}\mathbb{E}_{Y\sim K}\left[(H(u)-\one_{Y\leq u})^2\right]du\\
 &=& \int_{\mathbb{R}}\left(H(u)^2-2H(u)\mathbb{E}_{Y\sim K}\left[\one_{Y\leq u}\right]+\mathbb{E}_{Y\sim K}\left[\one_{Y\leq u}\right]\right)du\\
  &=& \int_{\mathbb{R}}\left(H(u)^2-2H(u)K(u)+K(u)\right)du.
\end{eqnarray*}
We use the fact that $a^2-2ab+b=(a-b)^2+b(1-b)$, with $a,b \in \mathbb{R}$, in the last equation to obtain the result.
\end{proof}
\begin{lemma}
\label{lem:expliciteformulaENT}
We consider local average estimators that take the form
\[
\hat{F}_{n,X}(Y)=\sum_{i=1}^{n}w_{ni}(X)\one_{Y_i\leq y}
\]
where $\left\{w_{ni}(X)\right\}$ are probability weights satisfying the conditions $\sum_{i=1}^{n}w_{ni}(X)=1$ and $w_{ni}(X)\geq 0$ for $i=1,\ldots,n$. Then, 
\[
\ent(\hat{F}_{n,X})=\sum_{i=1}^{n}\sum_{j=1}^{n} w_{ni}(X)w_{nj}(X)(Y_j-Y_i)\one_{Y_i< Y_j}.
\]
\end{lemma}

\begin{proof}
The definition of the weights $\left\{w_{ni}(X)\right\}$ ensures that
\[
1-\hat{F}_{n,X}(Y)=\sum_{j=1}^{n}w_{nj}(X)\one_{Y_j>y}.
\]
From this, we express $\ent(\hat{F}_{n,X})$ as
\begin{eqnarray*}
\ent(\hat{F}_{n,X})= \sum_{i=1}^{n}\sum_{j=1}^{n} w_{ni}(X)w_{nj}(X)\int_{\mathbb{R}}\one_{Y_i\leq y}\one_{Y_j>y}dy.
\end{eqnarray*}
The integration over \( y \) can be performed by analyzing the indicator functions $\one_{Y_i \leq y}$ and  $\one_{Y_j > y}$. If $Y_i \leq y$ and $Y_j > y$, this implies that $y \in (Y_i, Y_j)$, for $i\neq j$ 
and that the integral over $y$ is non-zero only within this interval. This leads us to compute
\[
\int_{\mathbb{R}} \one_{Y_i \leq y} \cdot \one_{Y_j > y} \, dy = \int_{Y_i}^{Y_j} 1 \, dy = Y_j - Y_i, \quad \text{for } Y_i < Y_j.
\]
Thus, the result follows.
\end{proof}
\section{Proof of Section~\ref{sec:generalframework}}
\begin{proof}[Proof of Proposition~\ref{prop:optimalpredictor}]
    By definition of $\mathcal{R}_{\lambda}(\Gamma_F)$ and $\crps$, we have
    \begin{eqnarray*}
        \mathcal{R}_{\lambda}(\Gamma_F) &=& \mathbb{E}\left[\int_{\mathbb{R}}(F_{X}(u)-\one_{\{Y\leq u\}})^2du \one_{\{\Gamma_F(X)\neq \re\}}\right]+\lambda r(\Gamma_F)\\
        &=& \mathbb{E}\left[\Div(F_X,F^{*}_{X}) \one_{\{\Gamma_F(X)\neq \re\}}\right]+\mathbb{E}\left[\crps(F^{*}_{X},Y)\one_{\{\Gamma_F(X)\neq \re\}}\right]\\
        && +2 \mathbb{E}\left[\int_{\mathbb{R}}(F_{X}(u)-F^{*}_X(u))(F^{*}_X(u)-\one_{\{Y\leq u\}}) du \one_{\{\Gamma_F(X)\neq \re\}}\right]+\lambda r(\Gamma_F).
    \end{eqnarray*}
Using Fubini's theorem we get
    \begin{eqnarray*}
         &&\mathbb{E}\left[\int_{\mathbb{R}}(F_{X}(u)-F^{*}_X(u))(F^{*}_X(u)-\one_{\{Y\leq u\}}) du \one_{\{\Gamma_F(X)\neq \re}\}\right]\\
         &=& \int_{\mathbb{R}}\mathbb{E}\left[(F_{X}(u)-F^{*}_X(u))(F^{*}_X(u)-\one_{\{Y\leq u\}}) \one_{\{\Gamma_F(X)\neq \re\}} \right]du \\
         &=& \int_{\mathbb{R}}\mathbb{E}\left[\mathbb{E}\left[(F_{X}(u)-F^{*}_X(u))(F^{*}_X(u)-\one_{\{Y\leq u\}})  \one_{\{\Gamma_F(X)\neq \re\}}\big|X \right]\right]du\\
         &=& \int_{\mathbb{R}}\mathbb{E}\left[(F_{X}(u)-F^{*}_X(u))(F^{*}_X(u)-\mathbb{E}\left[\one_{\{Y\leq u\}}  \big|X \right])\one_{\{\Gamma_F(X)\neq \re\}}\right]du =0.
    \end{eqnarray*}
    Then, the risk $\mathcal{R}_{\lambda}$ can be written as thanks to Lemma~\ref{lem:entropy}
   \begin{eqnarray}
        \mathcal{R}_{\lambda}(\Gamma_F) &=& \mathbb{E}\left[\int_{\mathbb{R}}(F_{X}(u)-\one_{\{Y\leq u\}})^2du \one_{\{\Gamma_F(X)\neq \re\}}\right]+\lambda r(\Gamma_F)\nonumber\\
        &=& \mathbb{E}\left[\left(\Div(F_X,F^{*}_{X}) +\crps(F^{*}_{X},Y)-\lambda\right)\one_{\{\Gamma_F(X)\neq \re\}}\right] +\lambda \nonumber\\
        &=& \mathbb{E}\left[\mathbb{E}\left[\left(\Div(F_X,F^{*}_{X}) +\crps(F^{*}_{X},Y)-\lambda\right)\one_{\{\Gamma_F(X)\neq \re\}}|X\right]\right] +\lambda \nonumber\\
         &=& \mathbb{E}\left[\left(\Div(F_X,F^{*}_{X}) +\ent(F^{*}_{X})-\lambda\right)\one_{\{\Gamma_F(X)\neq \re\}}\right] +\lambda .
         \label{eq: riskGamma}
    \end{eqnarray}
  At event $\{\Gamma_F(X)\neq \re\}$, it is evident that the minimum of the mapping $F \mapsto \Div(F_X,F^{*}_{X}) +\ent(F^{*}_{X})-\lambda $ occurs at $F(\cdot)=F^{*}_X(\cdot)$. Therefore, we consider the minimization over all predictors with reject option  $\Gamma$ of 
  \[
  \Gamma \mapsto \mathbb{E}\left[\left(\ent(F^*_X)-\lambda\right)\one_{\{\Gamma_F(X)\neq \re\}}\right]+\lambda.
  \]
  We can conclude that $\{\Gamma_F(X)\neq \re\}=\{\ent(F^*_X)\leq \lambda\}$.
\end{proof}
\begin{proof}[Proof of Proposition~\ref{prop:propertiesoptimalpredictor}]
Let $0<\lambda \leq \lambda'$. We remark that 
\[
\{\Gamma_{\lambda}^{*}(X)\neq \re\}=\{\ent(F^*_X)\leq \lambda\}\subset
\{\Gamma_{\lambda'}^{*}(X)\neq \re\}=\{\ent(F^*_X)\leq \lambda'\}.
\]
Then, we can deduce that $r(\Gamma_{\lambda'}^{*})\leq r(\Gamma_{\lambda}^{*})$. To prove the second inequality, we denote by $r_\lambda=\mathbb{P}(\Gamma_F(X)\neq \re)$. According to the definition of the error rate, and Lemma~\ref{lem:entropy} we obtain 
\begin{eqnarray*}
\err(\Gamma^*_\lambda)-\err(\Gamma^*_{\lambda'})&=&\frac{1}{r_\lambda}\mathbb{E}\left[\crps(F^*_X,Y)\one_{\{\ent(F^*_X)\leq \lambda\}}\right]-\frac{1}{r_{\lambda'}}\mathbb{E}\left[\crps(F^*_X,Y)\one_{\{\ent(F^*_X)\leq \lambda'\}}\right]\\
&=&\frac{1}{r_\lambda}\mathbb{E}\left[\mathbb{E}\left[\crps(F^*_X,Y)\one_{\{\ent(F^*_X)\leq \lambda\}}|X\right]\right]-\frac{1}{r_{\lambda'}}\mathbb{E}\left[\mathbb{E}\left[\crps(F^*_X,Y)\one_{\{\ent(F^*_X)\leq \lambda'\}}|X\right]\right]\\
&=&\left(\frac{1}{r_\lambda}-\frac{1}{r_{\lambda'}}\right)\mathbb{E}\left[\ent(F^*_X)\one_{\{\ent(F^*_X)\leq \lambda\}}\right]\\
&& -\frac{1}{r_{\lambda'}}\mathbb{E}\left[\ent(F^*_X)\left(\one_{\{\ent(F^*_X)\leq \lambda'\}}-\one_{\{\ent(F^*_X)\leq \lambda\}}\right)\right].\\
%&=&\left(\frac{1}{r_\lambda}-\frac{1}{r_{\lambda'}}\right)\mathbb{E}\left[\crps(F^*_X,Y)\mathds{1}_{\{\crps(F^*_X,Y)\leq \lambda\}}\right]\\
%&& -\frac{1}{r_{\lambda'}}\mathbb{E}\left[\crps(F^*_X,Y)\left(\mathds{1}_{\{\crps(F^*_X,Y)\leq \lambda'\}}-\mathds{1}_{\{\crps(F^*_X,Y)\leq \lambda\}}\right)\right]
\end{eqnarray*}
For the first term on the right of the previous equation, it is clear that
\begin{eqnarray*}
\left(\frac{1}{r_\lambda}-\frac{1}{r_{\lambda'}}\right)\mathbb{E}\left[\ent(F^*_X)\one_{\{\ent(F^*_X)\leq \lambda\}}\right]\leq \lambda\left(\frac{1}{r_\lambda}-\frac{1}{r_{\lambda'}}\right)\mathbb{E}\left[\one_{\{\ent(F^*_X)\leq \lambda\}}\right]= \lambda\left(1-\frac{r_\lambda}{r_{\lambda'}}\right).
\end{eqnarray*}
For the second term, we have
\begin{eqnarray*}
\frac{1}{r_{\lambda'}}\mathbb{E}\left[\ent(F^*_X)\left(\one_{\{\ent(F^*_X)\leq \lambda'\}}-\one_{\{\ent(F^*_X)\leq \lambda\}}\right)\right]&=& \frac{1}{r_{\lambda'}}\mathbb{E}\left[\ent(F^*_X)\one_{\{\lambda\leq \ent(F^*_X)\leq \lambda'\}}\right]\\
&\geq & \frac{\lambda}{r_{\lambda'}}\mathbb{P}\left(\lambda\leq \ent(F^*_X)\leq \lambda'\right)\\
&=&  \lambda\left(1-\frac{r_\lambda}{r_{\lambda'}}\right).
\end{eqnarray*}
Then, we can conclude that 
\[
\err(\Gamma^*_\lambda)-\err(\Gamma^*_{\lambda'})\leq 0.
\]
\end{proof}
\begin{proof}[Proof of Proposition~\ref{prop:optimalrule}]
First, note that for any $\varepsilon \in (0,1)$, if we define $\lambda_{\varepsilon}=G^{-1}_{\ent}(1-\varepsilon)$ where $G^{-1}_{\ent}$ is the generalised inverse of the cumulative distribution $G_{\ent}$, then the optimal predictor $\Gamma_{\lambda}$ given in Proposition~\ref{prop:optimalpredictor} with $\lambda=\lambda_{\varepsilon}$ satisfies the following, thanks to the properties of the quantile functions~\citep{Embrechts_Hofert13}:
\[
r(\Gamma^*_{\lambda_{\varepsilon}})= \mathbb{P}\left(\ent(F^{*}_{X}) > \lambda_{\varepsilon}\right)= \mathbb{P}\left(G_{\ent}(\ent(F^{*}_{X}))\geq 1-\varepsilon\right)=\varepsilon,
\]
and for $\varepsilon'\leq \varepsilon$, it follows that $\lambda_{\varepsilon}\leq \lambda_{\varepsilon'}$. Thus, by Proposition~\ref{prop:propertiesoptimalpredictor}, we have
\[
\err(\Gamma_{\lambda_{\varepsilon}}^*) \leq \err(\Gamma_{\lambda_{\varepsilon'}}^*).
\]
We now aim to prove that any predictor $\Gamma_F$ with $r(\Gamma_F)=\varepsilon'$ and $\varepsilon'\leq \varepsilon$ satisfies
$\err(\Gamma_{\lambda_\varepsilon}^*) \leq \err(\Gamma_{F})$. To do this, consider $\Gamma_{\lambda_{\varepsilon'}}^*$ with $\lambda_{\varepsilon'}=G^{-1}_{\ent}(1-\varepsilon')$. Thanks to the optimality of $\Gamma_{\lambda_{\varepsilon'}}^*$ (see Proposition~\ref{prop:optimalpredictor}), we get
\[
\err(\Gamma_F)-\err(\Gamma_{\lambda_{\varepsilon'}}^*)=\frac{1}{1-\varepsilon'}\left(\mathcal{R}_{\lambda_{\varepsilon'}}(\Gamma_F)-\mathcal{R}_{\lambda_{\varepsilon'}}(\Gamma_{\lambda_{\varepsilon'}}^*)\right)\geq 0.
\]
Therefore, we can deduce that 
\[
\err(\Gamma_{\lambda_{\varepsilon}}^*) \leq \err(\Gamma_{\lambda_{\varepsilon'}}^*)\leq \err(\Gamma_F),
\]
and we conclude the result.

\end{proof}

\begin{proof}[Proof of Proposition~\ref{prop:excessrisk}]
The risk $\mathcal{R}_{\lambda}(\Gamma_F)$ can be expressed as 
\[
 \mathcal{R}_{\lambda}(\Gamma_F)=  \mathbb{E}_X\left[\Div(F_X,F^{*}_{X})\one_{\{\Gamma_F(X)\neq \re\}}\right]+\mathbb{E}_X\left[(\ent(F_{X}^{*})-\lambda_{\varepsilon})\one_{\{\Gamma_F(X)\neq \re\}}\right]+\lambda_{\varepsilon}.
\]
Then we can deduce that 
\[
\mathcal{E}_{\lambda_\varepsilon}(\Gamma_F)=  \mathbb{E}_X\left[\Div(F_X,F^{*}_{X})\one_{\{\Gamma_F(X)\neq \re\}}\right]+\mathbb{E}_X\left[(\ent(F^*_X)-\lambda_{\varepsilon})\left(\one_{\{\Gamma_F(X)\neq \re\}}-\one_{\{\Gamma_{\varepsilon}^{*}(X)\neq \re\}}\right)\right].
\]
We have the equality between events $\{\Gamma_{\varepsilon}^{*}(X)\neq \re\}=\{\ent(F_{X}^{*})-\lambda_{\varepsilon}\leq 0\}$. Therefore we use the fact that $\one_{\{\Gamma_F(X)\neq \re\}}-\one_{\{\Gamma_{\varepsilon}^{*}(X)\neq \re\}}=\sgn(\ent(F_{X}^{*})-\lambda_{\varepsilon})$ to conclude the result.
\end{proof}
\section{Proof of Section~\ref{sec:statisticalguarantees}}
We begin by introducing a randomized pseudo-oracle with a precise rejection rate equal to $\varepsilon$, which is given as follows:

\[
   \tilde{\Gamma}_{\varepsilon}(X,\zeta)=\begin{cases}
  \hat{F}_{n,X}  & \text{if} \;\;\widehat{\ent}_{\zeta}(F^{*}_{X}) \leq \tilde{\lambda}_\varepsilon\\
  \re     & \text{otherwise} \enspace 
  \end{cases}
  =\begin{cases}
  \hat{F}_{n,X}  & \text{if} \;\; G_{\widehat{\ent}_{\zeta}}(\widehat{\ent}_{\zeta}(F^{*}_{X}) )\leq 1-\varepsilon,\\
  \re     & \text{otherwise} \enspace .
  \end{cases}
\]
where $G_{\widehat{\ent}_{\zeta}}(\cdot )=\mathbb{P}_{X,\zeta}\left(\widehat{\ent}_{\zeta}(F^{*}_{X})\leq \cdot |\mathcal{D}_n\right)$ and $\tilde{\lambda}_\varepsilon=G^{-1}_{\widehat{\ent}_{\zeta}}(1-\varepsilon)$ . Before proving the results of Section~\ref{sec:statisticalguarantees}, we provide the following result.
\begin{proposition}
Let 
\[
I_{\varepsilon}=\mathbb{E}\left[\big|\one_{\{\hat{\Gamma}_{\varepsilon}(X)\neq \re\}}-\one_{\{\tilde{\Gamma}_{\varepsilon}(X)\neq \re\}}\big|\right].
\]
Then, there exists a constant $C>0$ such that
\[
I_{\varepsilon}\leq C N^{-1/2}.
\]
\label{prop:Ivarepsilon}
\end{proposition}
\begin{proof}
The quantity $I_\varepsilon$ can be written as 
\begin{eqnarray*}
I_\varepsilon &=& \mathbb{E}\left[\big|\one_{\{\hat{\Gamma}_{\varepsilon}(X)\neq \re\}}-\one_{\{\tilde{\Gamma}_{\varepsilon}(X)\neq \re\}}\big|\right]\\
&=& \mathbb{E}\left[\big|\one_{\{\hat{G}_{\widehat{\ent}_{\zeta}}(\widehat{\ent}_{\zeta}(F^{*}_{X}))\geq 1-\varepsilon\}}-\one_{\{G_{\widehat{\ent}_{\zeta}}(\widehat{\ent}_{\zeta}(F^{*}_{X}))\geq 1-\varepsilon\}}\big|\right]\\
&=& \mathbb{E}\left[\one_{\{|\hat{G}_{\widehat{\ent}_{\zeta}}(\widehat{\ent}_{\zeta}(F^{*}_{X}))-G_{\widehat{\ent}_{\zeta}}(\widehat{\ent}_{\zeta}(F^{*}_{X}))|\geq |G_{\widehat{\ent}_{\zeta}}(\widehat{\ent}_{\zeta}(F^{*}_{X}))- 1+\varepsilon|\}}\right]\\
&=& \mathbb{P}\left(|\hat{G}_{\widehat{\ent}_{\zeta}}(\widehat{\ent}_{\zeta}(F^{*}_{X}))-G_{\widehat{\ent}_{\zeta}}(\widehat{\ent}_{\zeta}(F^{*}_{X}))|\geq |G_{\widehat{\ent}_{\zeta}}(\widehat{\ent}_{\zeta}(F^{*}_{X}))- 1+\varepsilon|\right).
\end{eqnarray*}
Let $\beta >0$. Then, we can deduce that 
\begin{equation}
I_\varepsilon \leq \mathbb{P}\left(|G_{\widehat{\ent}_{\zeta}}(\widehat{\ent}_{\zeta}(F^{*}_{X}))- 1+\varepsilon|\leq \beta\right)+\mathbb{P}\left(|\hat{G}_{\widehat{\ent}_{\zeta}}(\widehat{\ent}_{\zeta}(F^{*}_{X}))-G_{\widehat{\ent}_{\zeta}}(\widehat{\ent}_{\zeta}(F^{*}_{X}))|\geq \beta\right).
\label{eq:Iepsilon}
\end{equation}
By construction of $\widehat{\ent}_\zeta$, conditionally on $\mathcal{D}_n$, the random variable $G_{\widehat{\ent}_{\zeta}}(\widehat{\ent}_{\zeta}(F^{*}_{X}))~\sim \mathcal{U}([0,1])$. Then we can deduce that
\begin{equation}
\mathbb{P}\left(|G_{\widehat{\ent}_{\zeta}}(\widehat{\ent}_{\zeta}(F^{*}_{X}))- 1+\varepsilon|\leq \beta\right)=2\beta.
\label{eq:term1Iepsion}
\end{equation}
For the second term on the right side of Equation~\eqref{eq:Iepsilon}, we apply Lemma~\ref{lem:DKW} to $Z_i=\widehat{\ent}_{\zeta_i}(F^{*}_{X_i})$. These random variables are i.i.d. and real-valued, conditionally on $\mathcal{D}_n$. Thus, for all $\beta>0$, we have
\[
\mathbb{P}_{\mathcal{D}_N}\left(\sup_{z\in \mathbb{R}}\big|\hat{G}_{\widehat{\ent}_{\zeta}}(z)-G_{\widehat{\ent}_{\zeta}}(z)\big|\geq \beta\big|\mathcal{D}_n\right)\leq 2 \exp(-2N \beta^2).
\]
Therefore, we get 
\begin{eqnarray}
\mathbb{P}\left(|\hat{G}_{\widehat{\ent}_{\zeta}}(\widehat{\ent}_{\zeta}(X))-G_{\widehat{\ent}_{\zeta}}(\widehat{\ent}_{\zeta}(F^{*}_{X}))|\geq \beta\right)&\leq &\mathbb{P}\left(\sup_{z\in \mathbb{R}}\big|\hat{G}_{\widehat{\ent}_{\zeta}}(z)-G_{\widehat{\ent}_{\zeta}}(z)\big|\geq \beta\right)\nonumber\\
&=&\mathbb{E}\left[\mathbb{P}_{\mathcal{D}_N}\left(\sup_{z\in \mathbb{R}}\big|\hat{G}_{\widehat{\ent}_{\zeta}}(z)-G_{\widehat{\ent}_{\zeta}}(z)\big|\geq \beta\big|\mathcal{D}_n\right)\right]\nonumber\\
&\leq & 2 \exp(-2N \beta^2).
\label{eq:term2Iepsion}
\end{eqnarray}
Since the two conditions, Equation~\eqref{eq:term1Iepsion} and Equation~\eqref{eq:term2Iepsion}, of the Lemma~\ref{lem:DH20} are satisfied, applying the lemma yields the desired result.
\end{proof}
\begin{proof}[Proof of Proposition~\ref{prop:distributionfree}] Since $r(\tilde{\Gamma}_{\varepsilon})=\varepsilon$, we have that
\[
\mathbb{E}\left[|r(\hat{\Gamma}_{\varepsilon})-\varepsilon|\right]=\mathbb{E}\left[|r(\hat{\Gamma}_{\varepsilon})-r(\tilde{\Gamma}_{\varepsilon})|\right]=I_{\varepsilon}.
\]
Thus, the proposition~\ref{prop:Ivarepsilon} yields the result.
\end{proof}
\begin{proof}[Proof of Theorem~\ref{thm:consistency}]
 We consider the following decomposition
\begin{equation}
\mathbb{E}\left[\mathcal{E}_{\lambda_\varepsilon}(\hat{\Gamma}_{\varepsilon})\right]=\mathbb{E}\left[\mathcal{R}_{\lambda_\varepsilon}(\hat{\Gamma}_{\varepsilon})-\mathcal{R}_{\lambda_\varepsilon}(\tilde{\Gamma}_{\varepsilon})\right]+\mathbb{E}\left[\mathcal{E}_{\lambda_\varepsilon}(\tilde{\Gamma}_{\varepsilon})\right].
\label{eq:decompThm1}
\end{equation}
\item \textbf{Step 1.} We begin by establishing a bound on the term $\mathbb{E}\left[\mathcal{R}_{\lambda_\varepsilon}(\hat{\Gamma}_{\varepsilon})-\mathcal{R}_{\lambda_\varepsilon}(\tilde{\Gamma}_{\varepsilon})\right]$. By using Equation~\eqref{eq: riskGamma}, we can deduce that 
\[
\mathcal{R}_{\lambda_\varepsilon}(\hat{\Gamma}_{\varepsilon})-\mathcal{R}_{\lambda_\varepsilon}(\tilde{\Gamma}_{\varepsilon})= \mathbb{E}_{X,\zeta}\left[\left(\Div(\hat{F}_{n,X},F^{*}_{X})+(\ent(F_{X}^{*})-\lambda_{\varepsilon})\right)\left(\one_{\{\hat{\Gamma}_\varepsilon(X,\zeta)\neq \re\}}-\one_{\{\tilde{\Gamma}_\varepsilon(X,\zeta)\neq \re\}}\right)\right].
\]
Thanks to Assumption~\ref{ass:boundedentropy}, the quantity $|\ent(F_{X}^{*})-\lambda_{\varepsilon}|$ is bounded by $M<\infty$. Thus,
\begin{eqnarray*}
\mathbb{E}\left[\mathcal{R}_{\lambda_\varepsilon}(\hat{\Gamma}_{\varepsilon})-\mathcal{R}_{\lambda_\varepsilon}(\tilde{\Gamma}_{\varepsilon})\right]&\leq & \mathbb{E}\left[|\mathcal{R}_{\lambda_\varepsilon}(\hat{\Gamma}_{\varepsilon})-\mathcal{R}_{\lambda_\varepsilon}(\tilde{\Gamma}_{\varepsilon})|\right]\\
&\leq & \mathbb{E}\left[\Div(\hat{F}_{n,X},F^{*}_{X})\right]+\mathbb{E}\left[|\ent(F_{X}^{*})-\lambda_{\varepsilon}||\one_{\{\hat{\Gamma}_\varepsilon(X,\zeta)\neq \re\}}-\one_{\{\tilde{\Gamma}_\varepsilon(X,\zeta)\neq \re\}}|\right]\\
&\leq & \mathbb{E}\left[\Div(\hat{F}_{n,X},F^{*}_{X})\right]+MI_{\varepsilon}.
\end{eqnarray*}
We apply Proposition~\ref{prop:Ivarepsilon} to finish the first step.
\item \textbf{Step 2.} We control the term $\mathbb{E}\left[\mathcal{E}_{\lambda_\varepsilon}(\tilde{\Gamma}_{\varepsilon})\right]$. Proposition~\ref{prop:excessrisk} gives the following result
\[
\mathbb{E}\left[\mathcal{E}_{\lambda_\varepsilon}(\tilde{\Gamma}_{\varepsilon})\right]\leq  \mathbb{E}\left[\Div(\hat{F}_{n,X},F^{*}_{X})\right]+\mathbb{E}\left[\big|\ent(F_{X}^{*})-\lambda_{\varepsilon}\big|\one_{\{\tilde{\Gamma}_{\varepsilon}(X)\Delta \Gamma_{\varepsilon}^{*}(X)\}}\right].
\]
In the event $\{\tilde{\Gamma}_{\varepsilon}(X)\Delta \Gamma_{\varepsilon}^{*}(X)\}$, we have two cases:
\item[•] Case $1$: $\{\tilde{\Gamma}_{\varepsilon}(X)\setminus \Gamma_{\varepsilon}^{*}(X)\}$, we have that $\widehat{\ent}_{\zeta}(F^{*}_{X}) \leq \tilde{\lambda}_\varepsilon$ and $\ent(F^*_X) > \lambda_\varepsilon$. Thus, 
\begin{eqnarray*}
\big|\ent(F_{X}^{*})-\lambda_{\varepsilon}\big|=\ent(F_{X}^{*})-\lambda_{\varepsilon}=\ent(F_{X}^{*})-\widehat{\ent}_{\zeta}(F^{*}_{X})+\widehat{\ent}_{\zeta}(F^{*}_{X})-\tilde{\lambda}_\varepsilon+\tilde{\lambda}_\varepsilon-\lambda_{\varepsilon}.
\end{eqnarray*}
Since $\widehat{\ent}_{\zeta}(X)-\tilde{\lambda}_\varepsilon\leq 0$, we get 
\[
\mathbb{E}\left[\big|\ent(F_{X}^{*})-\lambda_{\varepsilon}\big|\one_{\{\tilde{\Gamma}_{\varepsilon}(X)\setminus \Gamma_{\varepsilon}^{*}(X)\}}\right]=\mathbb{E}\left[\left((\ent(F_{X}^{*})-\widehat{\ent}_{\zeta}(F^{*}_{X}))+(\tilde{\lambda}_\varepsilon-\lambda_{\varepsilon})\right)\one_{\{\tilde{\Gamma}_{\varepsilon}(X)\setminus \Gamma_{\varepsilon}^{*}(X)\}}\right].
\]
\item[•] Case $2$: $\{\Gamma_{\varepsilon}^{*}(X)\setminus \tilde{\Gamma}_{\varepsilon}(X) \}$, we have that $\widehat{\ent}_{\zeta}(F^{*}_{X}) > \tilde{\lambda}_\varepsilon$ and $\ent(F^*_X) \leq  \lambda_\varepsilon$. Thus, 
\begin{eqnarray*}
\big|\ent(F_{X}^{*})-\lambda_{\varepsilon}\big|=\lambda_{\varepsilon}-\ent(F_{X}^{*})=\lambda_{\varepsilon}-\tilde{\lambda}_\varepsilon+\tilde{\lambda}_\varepsilon-\widehat{\ent}_{\zeta}(F^{*}_{X})+\widehat{\ent}_{\zeta}(F^{*}_{X})-\ent(F_{X}^{*}).
\end{eqnarray*}
Using the fact that $\widehat{\ent}_{\zeta}(F^{*}_{X})-\tilde{\lambda}_\varepsilon\leq 0$, we get 
\[
\mathbb{E}\left[\big|\ent(F_{X}^{*})-\lambda_{\varepsilon}\big|\one_{\{\Gamma_{\varepsilon}^{*}(X)\setminus \tilde{\Gamma}_{\varepsilon}(X)\}}\right]=\mathbb{E}\left[\left((\widehat{\ent}_{\zeta}(F^{*}_{X})-\ent(F_{X}^{*}))+(\lambda_{\varepsilon}-\tilde{\lambda}_\varepsilon)\right)\one_{\{\Gamma_{\varepsilon}^{*}(X)\setminus \tilde{\Gamma}_{\varepsilon}(X)\}}\right].
\]
Recall that $r(\tilde{\Gamma}_{\varepsilon})= r(\Gamma_{\varepsilon}^{*}) =\varepsilon$. Using the observations above we can deduce that
\begin{eqnarray*}
\mathbb{E}\left[\big|\ent(F_{X}^{*})-\lambda_{\varepsilon}\big|\one_{\{\tilde{\Gamma}_{\varepsilon}(X)\Delta \Gamma_{\varepsilon}^{*}(X)\}}\right]&\leq & \mathbb{E}\left[\big|\widehat{\ent}_{\zeta}(F^{*}_{X})-\ent(F_{X}^{*})\big|\right]+\big|\lambda_{\varepsilon}-\tilde{\lambda}_\varepsilon\big|\mathbb{E}\left[\one_{\{\tilde{\Gamma}_{\varepsilon}(X)\Delta \Gamma_{\varepsilon}^{*}(X)\}}\right]\\
&\leq & \mathbb{E}\left[\big|\widehat{\ent}_{\zeta}(F^{*}_{X})-\ent(F_{X}^{*})\big|\right]+\big|\lambda_{\varepsilon}-\tilde{\lambda}_\varepsilon\big|\left(r(\tilde{\Gamma}_{\varepsilon})- r(\Gamma_{\varepsilon}^{*})\right)\\
&=& \mathbb{E}\left[\big|\widehat{\ent}_{\zeta}(F^{*}_{X})-\ent(F_{X}^{*})\big|\right].
\end{eqnarray*}
Using the definition of $\widehat{\ent}_{\zeta}$ and $|\zeta|\leq u $ for $u>0$, we get 
\begin{eqnarray*}
\mathbb{E}\left[\big|\widehat{\ent}_{\zeta}(F_{X}^{*})-\ent(F_{X}^{*})\big|\right]\leq \mathbb{E}\left[\big|\ent(\hat{F}_{n,X})-\ent(F_{X}^{*})\big|\right]+u.
\end{eqnarray*}
Therefore,
\[
\mathbb{E}\left[\mathcal{E}_{\lambda_\varepsilon}(\tilde{\Gamma}_{\varepsilon})\right]\leq  \mathbb{E}\left[\Div(\hat{F}_{n,X},F^{*}_{X})\right]+\mathbb{E}\left[\big|\ent(\hat{F}_{n,X})-\ent(F_{X}^{*})\big|\right]+u.
\]
Merging the results of the \textbf{Step 1.} and \textbf{Step 2.} in Eq.~\eqref{eq:decompThm1} and we get the result.
\end{proof}
\section{Drawback of choosing $\lambda$}
\label{sec:drawbackogchoosinglambda}
We illustrate the properties of $\Gamma^{*}_{\lambda}$, as given in the Proposition~\ref{prop:propertiesoptimalpredictor}, using the concrete dataset and the DRF predictor (see Section~\ref{sec:numerical}), as shown in Figure~\ref{fig:ErrorRejet}. The curves display $\widehat{\err}(\hat{\Gamma}_{\lambda})$ (blue-solid line) and $\hat{r}(\hat{\Gamma}_{\lambda})$ (red-dashed line) as functions of $\lambda$.
\begin{figure}[!ht]
\begin{center}
\begin{tabular}{c}
\includegraphics[height = 0.35 \columnwidth]{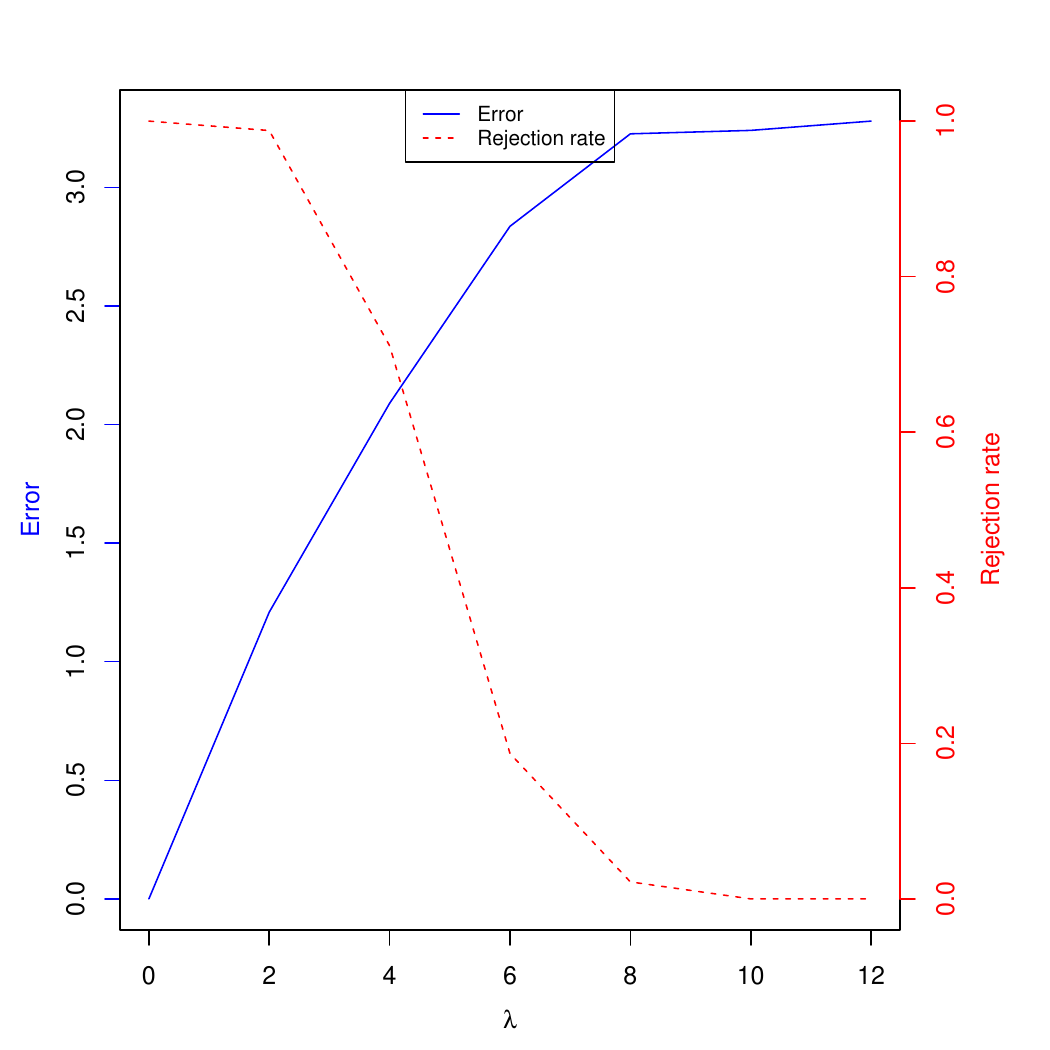} 
\end{tabular}
\end{center}
  \caption{ \label{fig:ErrorRejet} $\widehat{\err}\left(\hat{\Gamma}_{\lambda}\right)$ and $ \hat{r}\left(\hat{\Gamma}_{\lambda}\right)$ vs. $\lambda$.
   }
\end{figure}

\end{document}